\documentclass[11pt]{article}
 
\usepackage{ytableau}

\usepackage{amssymb,latexsym,amsmath,amsthm}
\usepackage{ytableau}
\usepackage[square,comma,numbers]{natbib}			
\usepackage[colorlinks = true, linkcolor = blue,citecolor = red]{hyperref} 	     
\usepackage{cleveref} 
\usepackage{tikz}
\allowdisplaybreaks

\newtheorem{theo}{Theorem}
\newtheorem{defi}[theo]{Definition}
\newtheorem{lemm}[theo]{Lemma}
\newtheorem{coro}[theo]{Corollary}
\newtheorem{prop}[theo]{Proposition}
\numberwithin{theo}{section} 
\numberwithin{equation}{section}

\newcommand{\C}{\mathbb{C}}

\newcommand{\N}{\mathbb{N}}
\newcommand{\Norm}{\operatorname{Norm}}
\newcommand{\Span}{\operatorname{span}}

\newcommand{\sgn}{\operatorname{sgn}}
\newcommand{\diag}{\operatorname{diag}}

\begin{document}

\begin{center}
{\Large \bf
Bases for infinite dimensional simple $\mathfrak{osp}(1|2n)$-modules respecting the branching $\mathfrak{osp}(1|2n)\supset \mathfrak{gl}(n)$
}
\vskip 1cm
{Asmus K. BISBO~$^\dag$
and Joris VAN DER JEUGT~$^\dag$}
\end{center}
\vskip 1cm
{$^\dag$~Department of Applied Mathematics, Computer Science and Statistics, Ghent University, Krijgs\-laan 281-S9, B-9000 Gent, Belgium}\\
{\href{mailto:Asmus.Bisbo@UGent.be}{Asmus.Bisbo@UGent.be}, 
\href{mailto:Joris.VanderJeugt@UGent.be}{Joris.VanderJeugt@UGent.be}
} 

\begin{abstract}
We study the effects of the branching $\mathfrak{osp}(1|2n)\supset \mathfrak{gl}(n)$ on a particular class of simple infinite-dimensional $\mathfrak{osp}(1|2n)$-modules $L(p)$ characterized by a positive integer $p$.
In the first part we use combinatorial methods such as Young tableaux and Young subgroups to construct a new basis for $L(p)$ that respects this branching and we express the basis elements explicitly in two distinct ways. First as monomials of negative root vectors of $\mathfrak{gl}(n)$ acting on the $\mathfrak{gl}(n)$-highest weight vectors in $L(p)$ and then as polynomials in the generators of $\mathfrak{osp}(1|2n)$ acting on the $\mathfrak{osp}(1|2n)$-lowest weight vector in $L(p)$.
In the second part we use extremal projectors and the theory of Mickelsson-Zhelobenko algebras to give new explicit constructions of raising and lowering operators related to the branching $\mathfrak{osp}(1|2n)\supset \mathfrak{gl}(n)$. We use the raising operators to give new expressions for the elements of the Gel'fand-Zetlin basis for $L(p)$ as monomials of operators from $U(\mathfrak{osp}(1|2n))$ acting on the $\mathfrak{osp}(1|2n)$-lowest weight vector in $L(p)$. 
We observe that the Gel'fand-Zetlin basis for $L(p)$ is related to the basis constructed earlier in the paper by a triangular transition matrix. We end the paper with a detailed example treating the case $n=3$.
	\end{abstract}

\section{Introduction}
\label{sec1}

In both the finite- and infinite-dimensional representation theory of Lie superalgebras many questions remain open.
In the particular case of infinite-dimensional modules of the orthosymplectic Lie superalgebra $\mathfrak{osp}(1|2n)$, where we have classification theorems \cite{Dobrev-Salom-2017,Dobrev-Zhang-2005} and some knowledge of character formulas \cite{Cheng-Zhang-2004}, there is only one class of simple modules that has been studied in detail. These are the simple lowest weight modules $L(p)$ of lowest weight $(\frac{p}{2},\dots,\frac{p}{2})$, for $p\in \N$, sometimes referred to as the paraboson Fock spaces. 
Character formulas and Gel'fand-Zetlin bases related to the branching $\mathfrak{osp}(1|2n)\supset \mathfrak{gl}(n)\supset \cdots \supset \mathfrak{gl}(1)$ were obtained for the modules $L(p)$ in \cite{Lievens-Stoilova-VanderJeugt-2008}. Currently no expressions are known for the elements of these Gel'fand-Zetlin bases in terms of polynomials of operators in $U(\mathfrak{osp}(1|2n))$ acting on the lowest weight vector of $L(p)$ (i.e.\ on the vacuum state of the paraboson Fock space). 
The search for a basis for $L(p)$ whose elements are given as polynomials of operators acting on a lowest weight vector has led to the discovery of the only other known basis for $L(p)$, see \cite{Bisbo-DeBie-VanderJeugt-2021}. 

The goal of this paper is two-fold. 
In the first part of the paper we identify the $\mathfrak{gl}(n)$-highest weight vectors of $L(p)$ and use this information to construct a new basis for $L(p)$ related to the branching $\mathfrak{osp}(1|2n)\supset \mathfrak{gl}(n)$. This basis restricts to a monomial basis on each of the simple $\mathfrak{gl}(n)$-submodules of $L(p)$. In \cite{Molev-Yakimova-2018}, such monomial bases were considered as PBW-parametrizations for various bases of simple $\mathfrak{gl}(n)$-modules. 
We proceed to construct a second expression for the elements of this new basis for $L(p)$ using Young subgroups to write each basis element as a polynomial in the generators of $\mathfrak{osp}(1|2n)$ acting on the lowest weight vector of $L(p)$.
In the second part of the paper we use the theory of extremal projectors to construct raising and lowering operators that act on the space of $\mathfrak{gl}(n)$-highest weight vectors of $L(p)$. These operators generate the Mickelsson-Zhelobenko algebra $Z(\mathfrak{osp}(1|2n),\mathfrak{gl}(n))$. 
Using these operators we obtain results regarding the action of $\mathfrak{osp}(1|2n)$ on the new basis for $L(p)$. Additionally, we apply these raising and lowering operators to obtain operator expressions for elements of the Gel'fand-Zetlin basis for $L(p)$. We end this part of the paper with the observation that the new basis for $L(p)$ constructed in this paper and the Gel'fand-Zetlin basis for $L(p)$ are related by a triangular transition matrix.
Throughout the paper we point out relevant connections to the theory of parabosons.

The paper is organized as follows. In Section \ref{sec2} we give an initial overview of the Lie superalgebra $\mathfrak{osp}(1|2n)$ and its subalgebra $\mathfrak{gl}(n)$. We present the classes of simple infinite-dimensional $\mathfrak{osp}(1|2n)$-modules $L(p)$ and simple finite-dimensional $\mathfrak{gl}(n)$-modules $V(\lambda+\frac{p}{2})$ that we will be working with throughout the paper. At the end of the section we introduce bases for the modules $V(\lambda+\frac{p}{2})$ and use them to construct a novel basis for $L(p)$. In Section \ref{sec3} we use combinatorial methods to obtain two distinct expressions for the elements of the basis for $L(p)$ which we introduced in Section~\ref{sec2}.
In Section \ref{sec4} we take a more general look at the branching $\mathfrak{osp}(1|2n)\supset \mathfrak{gl}(n)$ and construct novel expressions for the raising and lowering operators which generate the Mickelsson-Zhelobenko algebra $Z(\mathfrak{osp}(1|2n), \mathfrak{gl}(n))$ and which can be used to study the $\mathfrak{gl}(n)$-highest weight vectors in any given $\mathfrak{osp}(1|2n)$-module. We apply the properties of the Mickelsson-Zhelobenko algebra to the module $L(p)$ and use the raising operators to give new expressions for the elements of the Gel'fand-Zetlin basis for $L(p)$, relating this basis to the one we constructed in Section \ref{sec2} in the process.
In Section \ref{sec5} we provide a detailed treatment of the case $n=3$ while emphasizing connections to parabosons. Here we illustrate explicitly the main results of the paper and provide further interesting results. This includes the calculation of the matrix elements of the action of $\mathfrak{osp}(1|6)$ on the $\mathfrak{gl}(3)$-highest weight weight vectors in $L(p)$ and the explicit calculation of the transition matrix between the Gel'fand-Zetlin basis for $L(p)$ and the basis constructed in Section \ref{sec2}. 
With this transition matrix we are able to explicitly express the Gel'fand-Zetlin basis states of the paraboson Fock space $L(p)$ as polynomials in the parabosonic creation operators acting on the vacuum state.

\section{The Lie superalgebra $\mathfrak{osp}(1|2n)$ and its subalgebra $\mathfrak{gl}(n)$}
\label{sec2}

In this section we present the preliminary details regarding the Lie superalgebra $\mathfrak{osp}(1|2n)$ and its subalgebra $\mathfrak{gl}(n)$. This includes definitions of the algebras in terms of generators and relations, root systems and simple modules. We end the section with an initial discussion of bases for the simple $\mathfrak{osp}(1|2n)$-modules we will be studying.

\subsection{Definitions and root systems}
\label{sec2.1}

When it was introduced by Kac in \cite{Kac-1977}, the Lie superalgebra $\mathfrak{osp}(1|2n)$ was defined as a matrix algebra. For our purposes it will be more useful to instead use the equivalent definition in terms of generators and relations given by Ganchev and Palev in \cite{Ganchev-Palev-1980}.

\begin{theo}
The Lie superalgebra $\mathfrak{osp}(1|2n)$ is the superalgebra over $\C$ generated by the $2n$ odd elements $B_j^+$ and $B_j^-$, for $j\in\{1,\dots,n\}$, satisfying the relations
\begin{equation}
\label{sec2_eq_relations-osp}
[\{B_i^\xi,B_j^\eta\},B_l^\epsilon]=(\epsilon-\eta)\delta_{jl}B_i^\xi+ (\epsilon-\xi)\delta_{il}B_j^\eta,
	\end{equation}
for $i,j,l\in \{1,\dots,n\}$ and $\eta,\epsilon,\xi\in\{+,-\}$. Here $\pm$ is to be interpreted as $\pm 1$ in the algebraic expressions $\epsilon-\eta$ and $\epsilon-\xi$. Additionally the notations $[\cdot,\cdot]$ and $\{\cdot,\cdot\}$ are used for the commutator and anticommutator brackets.
	\end{theo}
	
There is a natural Hermitian anti-involution for $\mathfrak{osp}(1|2n)$ determined by $(B_i^\pm)^*=B_i^\mp$.
The Lie superalgebra $\mathfrak{osp}(1|2n)$ contains a subalgebra isomorphic to the Lie algebra $\mathfrak{gl}(n)$. 
\begin{prop}
The elements $E_{ij}:=\frac{1}{2}\{B_i^+,B_j^-\}$, for $i,j\in\{1,\dots,n\}$, form a basis for the subalgebra $\mathfrak{gl}(n)\subset \mathfrak{osp}(1|2n)$, satisfying the relations
\begin{equation}
\label{sec2_eq_relations-gl}
[E_{ij},E_{kl}]=  \delta_{jk}E_{il} - \delta_{il}E_{kj},\quad (i,j,k,l\in\{1,\dots,n\}),
\end{equation}
and conform to $E_{ij}^*=E_{ji}$.	
\end{prop}

We denote the Cartan subalgebra of $\mathfrak{osp}(1|2n)$ by $\mathfrak{h}$. This is simultaneously the Cartan subalgebra of $\mathfrak{gl}(n)$. The elements $E_{ii}$, for $i\in\{1,\dots,n\}$, form a basis for $\mathfrak{h}$. We denote the corresponding dual basis for $\mathfrak{h}^*$ by $\epsilon_i$, for $i\in\{1,\dots,n\}$.
The notation $(\mu_1,\dots,\mu_n)$ will be used to represent any weight $\mu=\sum_{i=1}^n\mu_i\epsilon_i\in\mathfrak{h}^*$.

The Lie superalgebra $\mathfrak{osp}(1|2n)$ has root system 
\begin{equation}
\big\{ \epsilon_i- \epsilon_j,\pm\epsilon_k \pm\epsilon_l, \pm\epsilon_m : i,j,k,l,m\in\{1,\dots,n\}, i\neq j,k\leq l \big\},
	\end{equation}
for which we choose the simple root system 
\begin{equation}
\{\epsilon_1-\epsilon_2,\dots,\epsilon_{n-1}-\epsilon_{n}, \epsilon_n\}.
	\end{equation}
Here $E_{ij}$ is a root vector of the root $e_i-e_j$, $\{B_k^\pm,B_l^\pm\}$ is a root vector of the root $\pm e_k\pm e_l$ and $B_m^\pm$ is a root vector of the root $\pm\epsilon_m$.

The root system and simple root system of $\mathfrak{gl}(n)$ are subsets of those for $\mathfrak{osp}(1|2n)$, specifically $\mathfrak{gl}(n)$ has root system 
\begin{equation}
\big\{ \epsilon_i- \epsilon_j : i,j\in\{1,\dots,n\}, i\neq j \big\}
	\end{equation}
for which we choose the simple root system 
\begin{equation}
\{\epsilon_1-\epsilon_2,\dots,\epsilon_{n-1}-\epsilon_{n}\}.
	\end{equation}
We let $\mathfrak{t}^+$ and $\mathfrak{t}^-$ denote the subalgebras of $\mathfrak{gl}(n)$ consisting of positive and negative root vectors respectively, in particular
\begin{equation}
\label{sec2_eq_positive-root-algebra}
\mathfrak{t}^+=\Span_\C\{E_{ij} : i<j\}
	\end{equation}
and
\begin{equation}
\label{sec2_eq_negative-root-algebra}
\mathfrak{t}^-=\Span_\C\{E_{ij} : i>j\}.
	\end{equation}

\subsection{Simple modules of $\mathfrak{osp}(1|2n)$ and $\mathfrak{gl}(n)$}
\label{sec2.2}

We will use the following notation for the simple modules of $\mathfrak{osp}(1|2n)$ and $\mathfrak{gl}(n)$ that will be relevant in this paper.

\begin{defi}
For any positive integer $p\in\N$ we let $L(p)$ denote the simple lowest weight module of $\mathfrak{osp}(1|2n)$ with lowest weight vector $v_0$ of weight $(\frac{p}{2},\dots,\frac{p}{2})$.

For any integral dominant weight $\mu\in\mathfrak{h}^*$ with $\mu_i\in\N_0$ and $\mu_1\geq \cdots \geq \mu_n$, we let $V(\mu+\frac{p}{2})$ denote the simple finite-dimensional highest weight module of $\mathfrak{gl}(n)$ with highest weight vector $\xi_\mu$ of weight $\mu+\frac{p}{2}:=(\mu_1+\frac{p}{2},\dots,\mu_n+\frac{p}{2})\in\mathfrak{h}^*$.
	\end{defi}

In the paper \cite{Ganchev-Palev-1980} it was observed that $L(p)$ describes the Fock space of $n$ parabosonic particles of order $p$. In this context the generators $B_j^+$ and $B_j^-$ of $\mathfrak{osp}(1|2n)$ may be considered as parabosonic creation and annihilation operators and $v_0$ may be considered as the vacuum state $|0\rangle$. In the case $p=1$, parabosons become usual bosons and $L(1)$ becomes the usual boson Fock space.
The relations \eqref{sec2_eq_relations-osp}, satisfied by $B_j^+$ and $B_j^-$, originate from the initial papers on para-particles \cite{Green-1953, Greenberg-Messiah-1965}. 

The paper \cite{Bisbo-DeBie-VanderJeugt-2021} presents a realization of $L(p)$ as a space of Clifford algebra valued polynomials on which $B_j^+$ and $B_j^-$ act as $p$-dimensional vector variables and Dirac operators respectively. Combinatorial methods are used to construct a basis for this realization of $L(p)$ whose elements are expressed as polynomials in the $B_j^+$-operators acting on the lowest weight vector. Other works related to this realization of $L(p)$ include \cite{Lavicka-Soucek-2017, Lavicka-2019, Salom-2013}.

In \cite{Lievens-Stoilova-VanderJeugt-2008} the module $L(p)$ is realized as a quotient of an induced lowest weight module of $\mathfrak{osp}(1|2n)$. This realization is used to obtain a character formula and Gel'fand-Zetlin (GZ) basis for $L(p)$. In Section \ref{sec4.4}, we will express the elements of the GZ-basis as monomials of operators in $U(\mathfrak{osp}(1|2n))$ acting on the lowest weight vector. Such expressions are hitherto unknown in the literature.

Given a partition $\lambda$, that is, a non-increasing sequence $(\lambda_1,\lambda_2,\dots)$ of non-negative integers, we let $\ell(\lambda)$ denote the length of $\lambda$. The length is the number of non-zero entries in $\lambda$.
We let $\mathbb{P}$ denote the set of partitions of length at most $n$. For any partition $\lambda\in \mathbb{P}$ it is thus only necessary to specify the first $n$ entries. For that reason we write $\lambda= (\lambda_1,\dots,\lambda_n)$, when $\lambda\in \mathbb{P}$. We will denote the conjugate partition of $\lambda$ by $\lambda'$. 
Any partition can be uniquely described as a combinatorial object known as a Young diagram. This is a collection of empty boxes organized into rows and columns. The Young diagram of a partition $\lambda$ has $\lambda_i$ boxes in the $i$'th row (counted from top to bottom). As an example consider the partition $\lambda=(4,3,1)$. The corresponding Young diagrams for $\lambda$ and for its conjugate are then
\begin{equation}
\ytableausetup{centertableaux,smalltableaux}
\lambda=\ydiagram{4,3,1}
\quad \text{ and } \quad
\lambda'=\ydiagram{3,2,2,1}.
	\end{equation}
See \cite{Macdonald-2015} for more information on partitions and Young diagrams.

It is known that the $\mathfrak{osp}(1|2n)$-module $L(p)$ decomposes into a multiplicity free sum of simple $\mathfrak{gl}(n)$-modules. 
This follows directly from the character fomula of $L(p)$ obtained in \cite{Lievens-Stoilova-VanderJeugt-2008}.
\begin{theo}
\label{sec2_theo_decomposition-into-gl(n)-modules}
Considered as $\mathfrak{gl}(n)$-module we have the following isomorphism
\begin{equation}
L(p)\cong \bigoplus_{\lambda\in\mathbb{P},\ \ell(\lambda)\leq p} V\left(\lambda+\frac{p}{2}\right).
	\end{equation}
	\end{theo}

\subsection{Semistandard Young tableaux and bases for $V\left(\lambda+\frac{p}{2}\right)$ and $L(p)$}
\label{sec2.3}

The simplest way to obtain a basis for $L(p)$ respecting the branching $\mathfrak{osp}(1|2n)\supset \mathfrak{gl}(n)$ is to choose a basis for each of the modules $V(\lambda+\frac{p}{2})$, for $\lambda\in\mathbb{P}$ with $\ell(\lambda)\leq p$. The union of these bases will then be a basis for $L(p)$. 
Doing this at each step of the branching $\mathfrak{osp}(1|2n)\supset \mathfrak{gl}(n)\supset \cdots \supset \mathfrak{gl}(1)$, is essentially how one obtains a GZ-basis for $L(p)$, see \cite{Lievens-Stoilova-VanderJeugt-2008}. In particular, a GZ-basis for $L(p)$ is a union of GZ-bases for the modules $V(\lambda+\frac{p}{2})$.

However, we will not at first be considering GZ-bases for the modules $V(\lambda+\frac{p}{2})$. 
In the first part of this paper we will instead be working with a simpler basis for $V(\lambda+\frac{p}{2})$, namely a Poincaré–Birkhoff–Witt (PBW) type basis, see Theorem \ref{sec2_theo_PBW-type-basis-u(n)}. 
Later we will discuss how the two bases for $L(p)$ obtained from the GZ- and PBW-type bases for $V(\lambda+\frac{p}{2})$ are related to each other, see Theorem \ref{sec4_theo_triangular-transition-matrix}.

To introduce the PBW-type basis for $V(\lambda+\frac{p}{2})$ we need the combinatorial objects known as semistandard (s.s.) Young tableaux. For a detailed reference see \cite{Macdonald-2015}.
Most of the notations and concepts that we shall use regarding s.s.\ Young tableaux appear there. The rest will be introduced when needed.
Let $\mathbb{Y}$ denote the set of s.s.\ Young tableaux with entries in $\{1,\dots,n\}$. 
Given $A\in \mathbb{Y}$, we let $\mu_A\in\N_0^n$ and $\lambda_A\in\mathbb{P}$ denote the weight and shape of $A$ respectively.
We furthermore let $\mathbb{Y}(\lambda)$ denote the set of s.s.\ Young tableaux in $\mathbb{Y}$ of shape $\lambda$. That is,
\begin{equation}
\mathbb{Y}(\lambda):=\{A\in\mathbb{Y}:\lambda_A=\lambda \}.
	\end{equation}
A s.s. Young tableau $A\in\mathbb{Y}$ can be regarded as a filling of the Young diagram $\lambda_A$ with numbers from $\{1,\dots,n\}$ such that the entries are non-decreasing from left to right along each row and strictly increasing from top to bottom along each column. The weight $\mu_A$ then determines the number of times each entry appears, $(\mu_A)_i$ being the number of $i$-entries. An example of a s.s. Young tableau of shape $\lambda_A=(4,3,1)$ and weight $\mu_A=(2,1,2,3)$ is
\begin{equation}
\ytableausetup{centertableaux,smalltableaux}
A=\ytableaushort{1134,244,3}.
	\end{equation}
We identify a partition $\lambda$ with the set of coordinates of the boxes in the corresponding Young diagram, letting the coordinate $(k,l)$ refer to the box in the $k$'th row (counted top to bottom) and in the $l$'th column (counted left to right). We then have
\begin{equation}
\lambda = \big\{ (k,l) : k\in\{1,\dots \ell(\lambda)\},\ l\in \{1,\dots, \lambda_k\} \big\}.
	\end{equation}
Given a s.s. Young tableau $A\in\mathbb{Y}(\lambda)$ we shall use the notation $A(k,l)$, for $(k,l)\in\lambda$, to refer to the entry of $A$ in the $k$'th row and $l$'th column.

The subset $\mathbb{Y}(\lambda)$ of $\mathbb{Y}$ indexes any basis of the module $V(\lambda+\frac{p}{2})$. This is because the character of $V(\lambda+\frac{p}{2})$ is a Schur function $s_\lambda(x_1,\dots,x_n)$, and $s_\lambda$ has an expression in terms of s.s. Young tableaux.
 In light of Theorem \ref{sec2_theo_decomposition-into-gl(n)-modules} this means that the set $\cup_{\lambda\in\mathbb{P},\ \ell(\lambda)\leq p}\mathbb{Y}(\lambda)$ indexes any basis of $L(p)$. 
The final piece of notation we shall need regarding s.s.\ Young tableaux is the exponent matrix. 
Given $A\in\mathbb{Y}$ the exponent matrix $\gamma_A\in M_{nn}(\N_0)$ of $A$ is the $n\times n$ lower triangular non-negative integer matrix with entries
\begin{equation}
\begin{split}
\label{sec2_eq_exponent-matrix}
(\gamma_A)_{ij}
	&:=\#\{  k\in\{1,\dots,\lambda_j\}: A(k,j)=i  \}\\
	&=\#\{ \text{ $i$'s in the $j$'th row of $A$ } \},
	\end{split}
	\end{equation}
for all $i,j\in \{1,\dots,n\}$.
Exponent matrices satisfy the following uniqueness property.
\begin{lemm}
\label{sec2_lemm_unique-tableau-exponent-matrix}
Let $A,B\in \mathbb{Y}$, then $\gamma_A=\gamma_B$ if and only if $A=B$.
	\end{lemm}

For any matrix $\gamma\in M_{nn}(\N_0)$ we consider the following element in the universal enveloping algebra $U(\mathfrak{t}^-)\subset U(\mathfrak{gl}(n))$:
\begin{equation}
\label{sec2_eq_E-gamma-matrix}
E^\gamma:=\prod_{k=2,\dots,n}^\rightarrow 
	(E_{k1}^{\gamma_{k1}}\cdots E_{k,k-1}^{\gamma_{k,k-1}}),
	\end{equation}
where the arrow refers to the ordering of the product ($k=2$ is leftmost and $k=n$ is rightmost).

With this notation in place we can now define the PBW-type bases for the modules $V(\lambda+\frac{p}{2})$. 
The following theorem comes from \cite{Molev-Yakimova-2018}.

\begin{theo}
\label{sec2_theo_PBW-type-basis-u(n)}
The simple unitary module $V(\lambda+\frac{p}{2})$ has basis
\begin{equation}
\big\{ E^{\gamma_A}\xi_\lambda : A\in \mathbb{Y}(\lambda) \big\},
	\end{equation}
where $\xi_\lambda$ is the highest weight vector of $V(\lambda+\frac{p}{2})$.
	\end{theo}
In the literature this basis can be found in \cite{Makhlin-2018,Molev-Yakimova-2018, Reineke-1997}. In \cite{Molev-Yakimova-2018, Reineke-1997} it appears as PBW parametrization of the GZ-basis \cite{Gelfand-Zetlin-1950,Molev-2002} and of the canonical (crystal) basis \cite{Lusztig-1990-1,Lusztig-1990-2,Kashiwara-1990}.

Using Theorem \ref{sec2_theo_decomposition-into-gl(n)-modules} and Theorem \ref{sec2_theo_PBW-type-basis-u(n)} the following result gives a new basis for $L(p)$ provided one can characterize the $\mathfrak{gl}(n)$-highest weight vectors of $L(p)$.
\begin{coro}
\label{sec2_coro_PBW-type-basis-osp}
The simple unitary module $L(p)$ contains a unique (up to a scalar) $\mathfrak{gl}(n)$-highest weight vector $\xi_\lambda$ of weight $\lambda+\frac{p}{2}$, for each $\lambda\in \mathbb{P}$ with $\ell(\lambda)\leq p$. Consequently, $L(p)$ has a basis
\begin{equation}
\big\{ E^{\gamma_A}\xi_{\lambda_A} : A\in \mathbb{Y},\ \ell(\lambda_A)\leq p \big\}.
	\end{equation}
	\end{coro}

One should note that in the notation of Corollary \ref{sec2_coro_PBW-type-basis-osp} the $\mathfrak{gl}(n)$-highest weight vector $\xi_0\in L(p)$, with the index $0$ referring to the empty s.s. Young tableau, is equal to the $\mathfrak{osp}(1|2n)$-lowest weight vector $v_0$ of $L(p)$ up to a non-zero scalar factor.

\section{Two explicit constructions of the new basis for $L(p)$}
\label{sec3}

In this section we consider the basis $\{E^{\gamma_A}\xi_{\lambda_A}\}$ for $L(p)$. We present two explicit constructions of the elements of this basis in terms of polynomials of operators from $U(\mathfrak{osp}(1|2n))$ acting on the lowest weight vector $v_0$ of $L(p)$.

We begin by constructing a set of vectors $\{\Omega_A\}$ in $L(p)$ that are defined as polynomials in the positive generators $B_j^+$ of $\mathfrak{osp}(1|2n)$ acting on $v_0$. These vectors are indexed by the set $\mathbb{E}$ of Young tableaux with entries in $\{1,\dots,n\}$. The set $\mathbb{E}$ of Young tableaux includes the set of $\mathbb{Y}$ of s.s. Young tableaux.
Among the vectors $\{\Omega_A\}$ we identify the $\mathfrak{gl}(n)$-highest weight vectors of $L(p)$, which we denote $\Omega_\lambda$, for $\lambda\in\mathbb{P}$ with $\ell(\lambda)\leq p$. We then use this information to prove the following identity which gives two explicit constructions of the elements of the basis $\{E^{\gamma_A}\xi_{\lambda_A}\}$ for $L(p)$ defined in Corollary \ref{sec2_coro_PBW-type-basis-osp}.
\begin{equation}
E^{\gamma_A}\Omega_{\lambda_A} 
	= \frac{(\lambda_A)_1!\dots(\lambda_A)_n!}{(\gamma_A)_{11}!\cdots (\gamma_A)_{nn}!}\Omega_A,
	\end{equation}
for $A\in\mathbb{Y}$ with $\ell(\lambda_A)\leq p$.

\subsection{The vectors $\Omega_A$}
\label{sec3.1}

Given $\lambda\in \mathbb{P}$ we let $\mathbb{E}$ denote the set of Young tableaux with entries in $\{1,\dots,n\}$ and let $\mathbb{E}(\lambda)$ denote the set of Young tableaux in $\mathbb{E}$ of shape $\lambda$.
Like the s.s.\ Young tableaux in $\mathbb{Y}(\lambda)$, a Young tableau $A\in \mathbb{E}(\lambda)$ can be defined as a filling of the Young diagram of shape $\lambda$ by the numbers in $\{1,\dots,n\}$, however contrary to the tableaux of $\mathbb{Y}(\lambda)$ we put no conditions on the placement of the entries of $A$.
The exponent matrix $\gamma_A$, shape $\lambda_A$ and weight $\mu_A$ of a Young tableau is defined in the same way as for a s.s. Young tableau. 
It should be noted that $\mathbb{Y}(\lambda)$ is a subset of $\mathbb{E}(\lambda)$.

To define the vectors $\Omega_A$, for $A\in \mathbb{E}(\lambda)$, we need a way to describe permutations of the rows and columns of a tableau in $\mathbb{E}(\lambda)$. 
To this end we consider the group of permutations
\begin{equation}
S_\lambda:= S_{\lambda_1}\times \dots \times S_{\lambda_{\ell(\lambda)}}
	\end{equation}
and let $\sgn(\sigma):=\sgn(\sigma_1)\cdots\sgn(\sigma_{\ell(\lambda)})$ denote the sign of $\sigma\in S_\lambda$. The group $S_\lambda$ is commonly known as the Young subgroup corresponding to the partition $\lambda$. Young subgroups are widely used in the classification of simple modules of the symmetric group, see \cite{James-1978}.

The group $S_\lambda$ acts on the set $\mathbb{E}(\lambda)$ by permuting the rows of Young tableaux. Given $A\in\mathbb{E}(\lambda)$ and $\tau\in S_\lambda$ we define the row permuted Young tableau $A^\tau$ to be the Young tableau with entries 
\begin{equation}
\label{sec3_eq_row-permutation}
A^\tau(k,l):=A(k,\tau_k(l)), \quad \big( (k,l)\in\lambda)\big).
	\end{equation}
Similarly the group $S_{\lambda'}$ acts on $\mathbb{E}(\lambda)$ by permuting the columns of Young tableaux. Given $A\in\mathbb{E}(\lambda)$ and $\sigma\in S_{\lambda'}$ we define the column permuted Young tableau $A_\sigma$ to be the Young tableau with entries 
\begin{equation}
A_\sigma(k,l):=A(\sigma_l(k),l), \quad \big( (k,l)\in\lambda)\big).
	\end{equation}
	
We can now define the vectors $\Omega_A$, for $A\in \mathbb{E}(\lambda)$. Let $A\in\mathbb{E}(\lambda)$ and consider the following element in $U(\mathfrak{osp}(1|2n))$
\begin{equation}
\label{sec3_eq_def-X_A}
B_A^+:= \prod_{l=1,\dots,\ell(\lambda')}^\rightarrow B_{A(1,l)}^+\cdots B_{A(\lambda_l',l)}^+
	\end{equation}
together with the vectors
\begin{equation}
\label{sec3_eq_def-omega_A}
\omega_A:=\sum_{\sigma\in S_{\lambda'}} \sgn(\sigma)B_{A_\sigma}^+v_0
	\end{equation}
and
\begin{equation}
\label{sec3_eq_def-Omega_A}
\Omega_A:=\sum_{\tau\in S_{\lambda}} \omega_{A^\tau}
	=\sum_{\tau\in S_{\lambda}}\sum_{\sigma\in S_{\lambda'}} 			\sgn(\sigma)B_{(A^\tau)_\sigma}^+v_0.
	\end{equation}
The arrow in \eqref{sec3_eq_def-X_A} refers to the ordering of the product ($l=1$ is leftmost and $l=\ell(\lambda')$ is rightmost).

The vectors $\omega_A$, for $A\in\mathbb{Y}$ with $\ell(\lambda_A)\leq p$, would form the basis for $L(p)$ considered in \cite{Bisbo-DeBie-VanderJeugt-2021} if we had defined $B_A^+$ with the opposite multiplication order $\prod_{l=1,\dots,\ell(\lambda')}^\leftarrow B_{A(1,l)}^+\cdots B_{A(\lambda_l',l)}^+$. Repeating the line of reasoning from that paper, it can be proven that the $\omega_A$'s defined here also provide a basis for $L(p)$.
The basis $\{\omega_A: A\in\mathbb{Y}, \ell(\lambda_A)\leq p\}$ for $L(p)$ does however not respect the branching $\mathfrak{osp}(1|2n)\supset \mathfrak{gl}(n)$. 
That is, an element of this basis is not in general an element of one of the simple $\mathfrak{gl}(n)$-module components $V(\lambda+\frac{p}{2})$ of $L(p)$, but is instead a linear combination of vectors from multiple of these components.

On the other hand, the vectors $\Omega_A$, for $A\in\mathbb{Y}$ with $\ell(\lambda_A)\leq p$, turn out to be scalar multiples of the vectors of the basis defined in Corollary \ref{sec2_coro_PBW-type-basis-osp}.

It should meanwhile be noted that the vectors $B_A^+v_0$, for $A\in\mathbb{Y}$ with $\ell(\lambda_A)\leq p$, can be proven not to form a basis for $L(p)$ in general.

\subsection{Correspondence between the vectors $\Omega_A$ and $E^{\gamma_A}\xi_{\lambda_A}$}
\label{sec3.2}

Looking back at the definition of the elements of the basis $\{E^{\gamma_A}\xi_{\lambda_A}\}$ from Corollary \ref{sec2_coro_PBW-type-basis-osp} we see that the exponent matrices $\gamma_A$ play an essential role in their definition.
To relate $\Omega_A$ with $E^{\gamma_A}\xi_{\lambda_A}$, we are thus compelled to take  a closer look at the dependence of $\Omega_A$ on the exponent matrix $\gamma_A$.

We begin by noting that any two tableaux $A,B\in\mathbb{E}(\lambda)$ have the same exponent matrix, that is, $\gamma_A=\gamma_B$, if and only if $A=B^\tau$ for some permutation $\tau\in S_\lambda$. 
We write $A\sim B$ if this is the case. This defines an equivalence relation on $\mathbb{E}(\lambda)$. We let $[A]$ denote the equivalence class of $A$.

Consider a partition $\lambda\in\mathbb{P}$ and a matrix $\gamma\in M_{nn}(\N_0)$ with column sum $\lambda$, that is, $\lambda_j=\sum_{i=1}^n \gamma_{ij}$, for all $j\in\{1,\dots,n\}$. We then let $D(\gamma)$ denote the unique tableau in $\mathbb{E}(\lambda)$ which has exactly $\gamma_{ij}$ $i$-valued entries in its $j$'th row, for any $i,j\in \{1,\dots,n\}$, and whose entries are non-decreasing from left to right along each row. For example, if
\begin{equation}
\gamma=\left(
\begin{matrix}
1&0&1\\
1&1&1\\
2&2&0
\end{matrix}
\right),
\quad \text{ then } \quad 
D(\gamma)=
\ytableausetup{smalltableaux,centertableaux}\ytableaushort{1233,233,12}.
	\end{equation}
The tableaux $D(\gamma)$ give a complete classification of the equivalence classes in $\mathbb{E}(\lambda)$.

\begin{lemm}
\label{sec3_lemm_exponent-matrix-tableau}
Let $\lambda\in\mathbb{P}$, then the map $\gamma\mapsto D(\gamma)$ defines a bijection from the matrices in $M_{nn}(\N_0)$ with column sum $\lambda$ to the equivalence classes of\ $\mathbb{E}(\lambda)$.
Furthermore, $\Omega_{A}=\Omega_{D(\gamma)}$, for all $A\in [D(\gamma)]$.
	\end{lemm}

At this point it is beneficial to remark some notational edge cases. Namely $\gamma_{D(\gamma)}=\gamma$ for all $\gamma\in M_{nn}(\N_0)$ with column sum $\lambda\in \mathbb{P}$, and $D(\gamma_B)\in [B]$, for all $B\in\mathbb{E}(\lambda)$.

The point of Lemma \ref{sec3_lemm_exponent-matrix-tableau} is that it lets us study $\Omega_{D(\gamma_A)}$ in place of $\Omega_A$, since $\Omega_{D(\gamma_A)}=\Omega_A$. The main benefit is the more explicit dependence on the exponent matrix $\gamma_A$.

As the following lemma illustrates, we can give a simple expression for the action of $E_{ij}$ on $\Omega_{D(\gamma)}$.
Shortly explained the lemma tells us that $E_{ii}$ acts on $\Omega_A$ with the scalar $(\frac{p}{2}+(\mu_A)_i)$ and that the action of $E_{ij}$, for $i\neq j$, on $\Omega_A$ is a sum over the $\Omega_B$'s for which $B$ can be obtained from $A$ by replacing a $j$-valued entry with an $i$-valued entry. For example,

\begin{equation}
\ytableausetup{smalltableaux,centertableaux}
E_{33}\Omega_{\ytableaushort{223,33}}
	= (\frac{p}{2}+3)\Omega_{\ytableaushort{223,33}}
\end{equation}
and
\begin{equation}
\ytableausetup{smalltableaux,centertableaux}
E_{43}\Omega_{\ytableaushort{223,33}}
	= \Omega_{\ytableaushort{223,43}}
	+ \Omega_{\ytableaushort{223,34}}
	+ \Omega_{\ytableaushort{224,33}}
	= 2\Omega_{\ytableaushort{223,34}}
	+ \Omega_{\ytableaushort{224,33}}
	\end{equation}

To state the lemma, we need to introduce notation for unit matrices.
For any $i,j\in \{1,\dots,n\}$ we let $e_{ij}\in M_{nn}(\N_0)$ denote the unit matrix with entries $(e_{ij})_{i'j'}:=\delta_{ii'}\delta_{jj'}$, for all $i',j'\in\{1,\dots,n\}$.

\begin{lemm}
\label{sec3_lemm_u(n)-action-on-Omega}
Let $\lambda\in \mathbb{P}$ and $\gamma\in M_{nn}(\N_0)$ with column sum $\lambda$. Then
\begin{equation}
E_{ij}\Omega_{D(\gamma)}= \delta_{ij}\frac{p}{2}\Omega_{D(\gamma)}+ \sum_{k=1}^n \gamma_{jk}\Omega_{D(\gamma+e_{ik}-e_{jk})},
	\end{equation}
for all $i,j\in\{1,\dots,n\}$.
	\end{lemm}
\begin{proof}
Let $A$ be a tableau in the equivalence class $[D(\gamma)]$. Then $A$ has shape $\lambda$ and \eqref{sec3_eq_def-omega_A} tells us that 
\begin{equation}
\label{sec3_eq_multi-bracket-description}
\omega_A = 
[B_{A(1,1)}^+,\dots,B_{A(\lambda_1',1)}^+]\cdots [B_{A(1,\ell(\lambda'))}^+,\dots,B_{A(\lambda_{\ell(\lambda')}',\ell(\lambda'))}^+]v_0,
	\end{equation}
where we use the following notation:
\begin{equation}
\label{sec3_eq_def-multi-bracket}
[B_{i_1}^+,\dots,B_{i_k}^+]
	:=
	\sum_{\sigma\in S_k} \sgn(\sigma) B_{i_{\sigma(1)}}^+\dots B_{i_{\sigma(k)}}^+.
	\end{equation}

Using the relation $[E_{ij},B_k^+]=\delta_{jk}B_i^+$ repeatedly we get
\begin{equation}
\label{sec3_eq_action-on-multi-bracket}
\begin{split}
\left[ E_{ij},[B_{A(1,l)}^+,\dots,B_{A(\lambda_l',l)}^+]\right]
	=\sum_{k=1}^{\lambda_l'} \delta_{j,A(k,l)}
	[B_{A(1,l)}^+,\dots,B_{A(k-1,l)}^+,B_i^+,B_{A(k+1,l)}^+,\dots,B_{A(\lambda_l',l)}^+]
	\end{split},
	\end{equation}
for any $l\in \{1,\dots,\ell(\lambda')\}$.
Keeping in mind that $E_{ij}v_0=\delta_{ij}\frac{p}{2}v_0$ we get the following identity through repeated use of \eqref{sec3_eq_action-on-multi-bracket}
\begin{equation}
E_{ij}\omega_A	=
	\delta_{ij}\frac{p}{2}\omega_A
	+\sum_{\substack{(k,l)\in\lambda,\\ A(k,l)=j}} 
	\omega_{A_{(k,l)\to i}},
	\end{equation}
where $A_{(k,l)\to i}$ is the tableau obtained by replacing the entry of $A$ at coordinate $(k,l)$ with an $i$. In particular, for any $(s,t)\in\lambda$,
\begin{equation}
A_{(k,l)\to i}(s,t)=
\begin{cases}
A(s,t), & \text{ if } (s,t)\neq (k,l),\\
i, & \text{ if } (s,t)=(k,l).
	\end{cases}
	\end{equation}
Recalling the action of $S_\lambda$ on $\mathbb{E}(\lambda)$ defined in \eqref{sec3_eq_row-permutation} we note that 
\begin{equation}
(A_{(k,l)\to i})^\tau = (A^\tau)_{(k,\tau_k^{-1}(l))\to i}.
	\end{equation}
A short calculation then implies
\begin{equation}
E_{ij}\omega_{A^\tau}=
	\delta_{ij}\frac{p}{2}\omega_A
	+\sum_{\substack{(k,l)\in\lambda,\\ A(k,l)=j}} 
	\omega_{(A_{(k,l)\to i})^\tau}.
	\end{equation}
Using \eqref{sec3_eq_def-Omega_A} we then get
\begin{equation}
E_{ij}\Omega_{A}=
	\delta_{ij}\frac{p}{2}\Omega_A
	+\sum_{\substack{(k,l)\in\lambda,\\ A(k,l)=j}} 
	\Omega_{A_{(k,l)\to i}}.
	\end{equation}
To get the lemma from this it is enough to recall that $\Omega_A=\Omega_{D(\gamma)}$ by Lemma \ref{sec3_lemm_exponent-matrix-tableau} and to realize that if $A(k,l)=j$, then $\Omega_{A_{(k,l)\to i}}=\Omega_{D(\gamma+e_{ik}-e_{jk})}$. 
	\end{proof}

An exciting consequence of Lemma \ref{sec3_lemm_u(n)-action-on-Omega} is that we can now identify the $\mathfrak{gl}(n)$-highest weight vectors in $L(p)$.

Given a partition $\lambda\in \mathbb{P}$, we consider the diagonal matrix $\gamma_\lambda\in M_{nn}(\N_0)$ whose entries are $(\gamma_\lambda)_{ij}:=\lambda_i\delta_{ij}$, for all $i,j\in\{1,\dots,n\}$. Alternatively we can write $\gamma_\lambda=\sum_{i=1}^n\lambda_i e_{ii}$.
The tableau $D(\gamma_\lambda)$ is then the s.s.\ Young tableau in $\mathbb{Y}(\lambda)$ with all $1$'s in the first row, all $2$'s in the second row and all $k$'s in the $k$'th row.
We shall make use of the shorthand notation $\Omega_{\lambda}:= \Omega_{D(\gamma_\lambda)}$ and similarly $\omega_{\lambda}:= \omega_{D(\gamma_\lambda)}$, noting here that $\Omega_\lambda=(\lambda_1!\cdots\lambda_n!)\omega_\lambda$.

\begin{prop}
\label{sec3_prop_highest-weight-vector}
Let $\lambda\in \mathbb{P}$. If $\ell(\lambda)\leq p$, then the vector $\Omega_{\lambda}$ is non-zero and a $\mathfrak{gl}(n)$-highest weight vector in $L(p)$ of weight $\frac{p}{2}+\lambda$. If $\ell(\lambda)>p$, then $\Omega_\lambda=0$.
This classifies all the $\mathfrak{gl}(n)$-highest weight vectors in $L(p)$, up to multiples by scalar factors.
	\end{prop}
\begin{proof}
Consider first the case $\ell(\lambda)\leq p$.
As remarked following \eqref{sec3_eq_def-Omega_A} the set $\{\omega_A:A\in\mathbb{Y}, \ell(\lambda_A)\}$ is a basis for $L(p)$. Since $\omega_\lambda$ is an element of this basis it follows that $\omega_\lambda\neq 0$. This in turn implies that $\Omega_\lambda\neq 0$.
We need to check two things, namely that $\Omega_{\lambda}$ has the right weight and that it is annihilated by $\mathfrak{t}^+$.
Both follow directly from Lemma \ref{sec3_lemm_u(n)-action-on-Omega} when we recall that the positive root vectors of $\mathfrak{gl}(n)$ are the $E_{ij}$'s with $i<j$. So the vectors $\Omega_{\lambda}$ with $\ell(\lambda)\leq p$ are $\mathfrak{gl}(n)$-highest weight vectors in $L(p)$. By Theorem \ref{sec2_theo_decomposition-into-gl(n)-modules}, any other $\mathfrak{gl}(n)$-highest weight vector in $L(p)$ must be a scalar multiple of one of these.

If $\ell(\lambda)>p$, then $\Omega_\lambda$ is annihilated by $\mathfrak{t}^+$. By Theorem \ref{sec2_theo_decomposition-into-gl(n)-modules} no $\mathfrak{gl}(n)$-highest weight vector of this weight exists in $L(p)$. We conclude that $\Omega_\lambda=0$.
	\end{proof}

Since we now have explicit constructions of the $\mathfrak{gl}(n)$-highest weight vectors in $L(p)$ we can give the first expression of the elements of the basis from Corollary \ref{sec2_coro_PBW-type-basis-osp} in terms of operators acting on $v_0$. Specifically the elements of this basis are $E^{\gamma_A}\Omega_{\lambda_A}$, for $A\in \mathbb{Y}$ with $\ell(\lambda_A)\leq p$.

In the special case where all $j$-valued entries of $D(\gamma)$ are located in the same row Lemma \ref{sec3_lemm_u(n)-action-on-Omega} can be made more precise.
\begin{lemm}
\label{sec3_lemm_u(n)-action-on-Omega-multi}
Let $\lambda\in \mathbb{P}$ and $\gamma\in M_{nn}(\N_0)$ with column sum $\lambda$. Let $j\in\{1,\dots,n\}$ and suppose there exists $k\in \{1,\dots,n\}$ such that $\gamma_{jk'}=0$ when $k'\neq k$. Then 
\begin{equation}
E_{ij}^m\Omega_{D(\gamma)}= \frac{\gamma_{jk}!}{(\gamma_{jk}-m)!} \Omega_{D(\gamma+m(e_{ik}-e_{jk}))},
	\end{equation}
for all $i\neq j$, and $m\leq \gamma_{jk}$. If $m>\gamma_{jk}$, then $E_{ij}^m\Omega_{D(\gamma)}=0$.
	\end{lemm}

We are now in a position to prove the main theorem of this section.

\begin{theo}
\label{sec3_theo_PBW-type-basis-vector-variable-case}
Let $A\in \mathbb{E}(\lambda)$ and suppose $\gamma=\gamma_A$ is lower triangular. Then

\begin{equation}
E^{\gamma} \Omega_\lambda=\frac{\lambda!}{\diag (\gamma)!}\Omega_{D(\gamma)}=\frac{\lambda!}{\diag (\gamma)!}\Omega_A
	\end{equation}
where $\diag (\gamma)!:=\gamma_{11}!\cdots\gamma_{nn}!$ and $\lambda!=\lambda_1!\cdots\lambda_n!$.
	\end{theo}
\begin{proof}
Note first that $\Omega_A=\Omega_{D(\gamma)}$. 
Recall that $\Omega_\lambda=\Omega_{D(\gamma_\lambda)}$, where $\gamma_\lambda=\sum_{i=1}^n\lambda_i e_{ii}$. In a similar way we can write
\begin{equation}
\gamma= \sum_{i=1}^n\sum_{j=1}^i \gamma_{ij}e_{ij}
	= \gamma_\lambda +\sum_{i=2}^n\sum_{j=1}^{i-1} \gamma_{ij}(e_{ij}-e_{jj}).
	\end{equation}

Proving the theorem is then a matter of consecutive application of the operators $E_{ij}^{\gamma_{ij}}$ in the order given by the definition of $E^\gamma$ in \eqref{sec2_eq_E-gamma-matrix}, noting at each successive step that the conditions of Lemma \ref{sec3_lemm_u(n)-action-on-Omega-multi} are satisfied. To this end it should be noted that $\gamma$ has column sum $\lambda$ since $A\in \mathbb{E}(\lambda)$.
In broad strokes we proceed as follows. First note that 
\begin{equation}
\begin{split}
E_{n1}^{\gamma_{n1}}\cdots E_{n,n-1}^{\gamma_{n,n-1}}\Omega_{D(\gamma_\lambda)}
	&=
	\frac{\lambda_{n-1}!}{(\lambda_{n-1}-\gamma_{n,n-1})!}
	E_{n1}^{\gamma_{n1}}\cdots E_{n,n-2}^{\gamma_{n,n-2}}
	\Omega_{D(\gamma_\lambda+\gamma_{n,n-1}(e_{n,n-1}-e_{n-1,n-1}))}
	\\&=
	\left(
	\prod_{j=1}^{n-1}\frac{\lambda_j!}{(\lambda_j-\gamma_{nj})!}
	\right)
	\Omega_{D(\gamma_\lambda+\sum_{j=1}^{n-1}\gamma_{nj}(e_{nj}-e_{jj}))}.
	\end{split}
	\end{equation}
Continuing in the same way we get
\begin{equation*}
\begin{split}
E^\gamma\Omega_{D(\gamma_\lambda)} 
	&= 
	\prod_{k=2,\dots,n}^\rightarrow 
	(E_{k1}^{\gamma_{k1}}\cdots E_{k,k-1}^{\gamma_{k,k-1}})
	\Omega_{D(\gamma_\lambda)} 
	\\&=
	\left(
	\prod_{j=1}^{n-1}\frac{\lambda_j!}{(\lambda_j-\gamma_{nj})!}
	\right)
	\prod_{k=2,\dots,n-1}^\rightarrow 
	(E_{k1}^{\gamma_{k1}}\cdots E_{k,k-1}^{\gamma_{k,k-1}})
	\Omega_{D(\gamma_\lambda+\sum_{j=1}^{n-1}\gamma_{nj}(e_{nj}-e_{jj}))}
	\\&=
	\left(
	\prod_{j=1}^{n-1}\frac{\lambda_j!}{(\lambda_j-\gamma_{nj})!}\frac{(\lambda_j-\gamma_{nj})!}{(\lambda_j-\gamma_{nj}-\gamma_{n-1,j})!}\cdots\frac{(\lambda_j-\sum_{i=j+2}^n \gamma_{ij})!}{\lambda_j-\sum_{i=j+1}^n \gamma_{ij})!}
	\right)
	\Omega_{D(\gamma_\lambda+\sum_{i=2}^n\sum_{j=1}^{i-1}\gamma_{ij}(e_{ij}-e_{jj}))}
	\\&=
	\frac{\lambda!}{\diag (\gamma)!}
	\Omega_{D(\gamma)}.
	\end{split}
	\end{equation*}
This proves the theorem.
	\end{proof}

We note here that if $A$ is a s.s.\ Young tableau, then $\gamma_A$ is automatically lower triangular and Theorem \ref{sec3_theo_PBW-type-basis-vector-variable-case} applies. This leads to the following corollary which states the two expressions we have obtained for the elements of the basis $\{E^{\gamma_A}\xi_{\lambda_A}\}$ for $L(p)$ defined in Corollary \ref{sec2_coro_PBW-type-basis-osp}.
\begin{coro}
\label{sec3_coro_PBW-bases-Omega_A}
The simple $\mathfrak{osp}(1|2n)$-module $L(p)$ has a basis consisting of the elements 
\begin{equation}
\label{sec3_eq_PBW-bases-Omega_A}
E^{\gamma_A} \Omega_{\lambda_A}=\frac{\lambda_A!}{\diag(\gamma_A)!}\Omega_A
	\end{equation}
for all $A\in\mathbb{Y}$ with $\ell(\lambda_A)\leq p$. 
	\end{coro}

The two expressions for the basis elements presented in Corollary \ref{sec3_coro_PBW-bases-Omega_A} may be described as follows. On the one hand the expression $E^{\gamma_A}\Omega_{\lambda_A}$ gives the basis elements as monomials in the negative root vectors $E_{ij}$, for $i>j$, of $\mathfrak{gl}(n)$ acting on $\mathfrak{gl}(n)$-highest weight vectors in $L(p)$. 
This expression makes it clear how the basis behaves under the branching $\mathfrak{osp}(1|2n)\supset \mathfrak{gl}(n)$.
On the other hand the expression $\frac{\lambda_A!}{\diag(\gamma_A)!}\Omega_A$ gives the basis elements as polynomials in the positive generators $B_j^+$ of $\mathfrak{osp}(1|2n)$ acting on the $\mathfrak{osp}(1|2n)$-lowest weight vector $v_0$.
Such expressions are useful in the context of parastatistics where $L(p)$ is identified with the paraboson Fock space, the $B_j^+$'s are interpreted as creation operators and $v_0$ as the vacuum. Using the second expression of the basis we can then present any state in the paraboson Fock space as a polynomial of creation operators acting on the vacuum.

An illustrative example of the expressions \eqref{sec3_eq_PBW-bases-Omega_A} can be found in Section \ref{sec5} were we describe the results of this paper in the case $n=3$.

To end this section we present the necessary and sufficient conditions on $\gamma$ for $D(\gamma)$ to be a s.s. Young tableau in $\mathbb{Y}(\lambda)$.
\begin{lemm}
\label{sec3_lemm_betweenness-like-condition}
The tableau $D(\gamma)$ is a s.s. Young tableau in $\mathbb{Y}(\lambda)$ if and only if $\gamma$ is a lower triangular matrix in $M_{nn}(\N_0)$ with column sum $\lambda\in \mathbb{P}$ and which satisfies 
\begin{equation}
\label{sec3_eq_betweenness-like-condition}
\sum_{k=i}^{j} \gamma_{ki}\geq \sum_{k=i+1}^{j+1}\gamma_{k,i+1},
	\end{equation}
for $1\leq i\leq j\leq n-1$. 
	\end{lemm}
\begin{proof}
The tableau $D(\gamma)$ has shape $\lambda$ if and only if $\gamma$ has column sum $\lambda$.
By definition $D(\gamma)$ only has entries in the set $\{1,\dots,n\}$ and the entries are non-decreasing from left to right along each row. 
It only remains to check that \eqref{sec3_eq_betweenness-like-condition} is a necessary and sufficient condition for the entries of $D(\gamma)$ to be increasing from top to bottom along each column.
Since the entries are non-decreasing along the rows, this is equivalent to requiring that for each pair $(i,j)$ with $1\leq i<j\leq n-1$, the number of entries in the $i$'th row of $D(\gamma)$ taking values $1,\dots,j$ is greater then the number of entries in the $(i+1)$'th row of $D(\gamma)$ taking values $1,\dots,j+1$. 
Recalling the definition of $D(\gamma)$, this happens precisely if 
\begin{equation}
\sum_{k=1}^{j} \gamma_{ki}\geq \sum_{k=1}^{j+1}\gamma_{k,i+1}.
	\end{equation}
By apply the assumption that $\gamma$ is lower triangular, we see that this is exactly the condition \eqref{sec3_eq_betweenness-like-condition}.
	\end{proof}

\section{Raising and lowering operators}
\label{sec4}

In this section we study the branching $\mathfrak{osp}(1|2n)\supset \mathfrak{gl}(n)$ in more detail. To do so, we will construct raising and lowering operators $z_j^\pm$ and $z_{ij}^\pm$, for $1\leq i\leq j\leq n$, that act on the space of $\mathfrak{gl}(n)$-highest weight vectors of any given $\mathfrak{osp}(1|2n)$-module. 
These raising and lowering operators generate an associative algebra $Z(\mathfrak{osp}(1|2n),\mathfrak{gl}(n))$ called the Mickelsson-Zhelobenko algebra. The literature on the application of raising and lowering operators and Mickelsson-Zhelobenko algebras to the theory of Lie superalgebras is limited and this work represents for the first time such constructions being applied to the branching $\mathfrak{osp}(1|2n)\supset\mathfrak{gl}(n)$. In particular, the identities obtained in Theorem \ref{sec4_theo_mickelsson-zhelobenko-generator} are entirely new.
In the context of Lie algebras, abstract algebras of raising and lowering operators where introduced by Mickelsson in \cite{Mickelsson-1973}. These algebras are commonly known as Mickelsson algebras. Zhelobenko expanded on the theory of Mickelsson algebras in \cite{Zhelobenko-1983,Zhelobenko-1989} using the extremal projector $\mathbf{p}$, introduced by Asherova, Smirnov and Tolstoy in \cite{Asherova-Smirnov-Tolstoy-1979}, to construct generators for the extension of the Mickelsson algebra by the field of fractions of the Cartan subalgebra. This extension is called the extended Mickelsson algebra or the Mickelsson-Zhelobenko algebra.
Many aspects of this theory extend to Lie superalgebras, which is the situation we are dealing with in this paper. Relevant papers on Lie superalgebraic applications include \cite{Molev-2011,Tolstoy-1985}. 


We begin the section with a brief explanation of the aforementioned concepts. Following that we present precise expressions of the operators $z_j^\pm$ and $z_{ij}^\pm$ that generate $Z(\mathfrak{osp}(1|2n),\mathfrak{gl}(n))$, see Theorem \ref{sec4_theo_mickelsson-zhelobenko-generator}.
We then apply the raising and lowering operators to the $\mathfrak{osp}(1|2n)$-module $L(p)$. By using the expressions of $z_j^\pm$ given in Theorem \ref{sec4_theo_mickelsson-zhelobenko-generator} we obtain formulas for the action of $\mathfrak{osp}(1|2n)$ on the 
vectors in $L(p)$ of the form $E^{\gamma}\Omega_{\lambda}$. 
We explain what these formulas yield about the matrix elements of the action of $\mathfrak{osp}(1|2n)$ on the elements of the basis $\{E^{\gamma_A}\Omega_{\lambda_A}\}$ for $L(p)$.

Finally we use the raising and lowering operators to express the elements of the GZ-basis for $L(p)$ as monomials of operators acting on the $\mathfrak{osp}(1|2n)$-lowest weight vector $v_0$ of $L(p)$, something that until now was missing in the literature.
Using results from \cite{Molev-Yakimova-2018} we conclude that there is a triangular transition matrix connecting the GZ-basis and the $\{E^{\gamma_A}\Omega_{\lambda_A}\}$ basis for $L(p)$, see Theorem \ref{sec4_theo_triangular-transition-matrix}. We briefly mention how such a transition matrix can be used to express any vector in the the GZ-basis as a polynomial in the generators $B_j^+$ of $\mathfrak{osp}(1|2n)$, for $j\in\{1,\dots,n\}$, acting on $v_0$. 
The entries of this transition matrix are calculated explicitly in the case $n=3$, this is done in Proposition~\ref{sec5_prop_transition-matrix-n=3}.

\subsection{The Mickelsson-Zhelobenko algebra $Z(\mathfrak{osp}(1|2n),\mathfrak{gl}(n))$}
\label{sec4.1}

To keep the notation in this section manageable, we introduce the shorthand notation $\mathfrak{g}:=\mathfrak{osp}(1|2n)$ and $\mathfrak{t}:=\mathfrak{gl}(n)$.
We recall that $\mathfrak{t}$ has triangular decomposition 
\begin{equation}
\mathfrak{t}= \mathfrak{t}^+\oplus \mathfrak{h} \oplus \mathfrak{t}^-,
	\end{equation}
where $\mathfrak{t}^+$ and $\mathfrak{t}^-$ denote the subalgebras spanned by the positive and negative roots vectors of $\mathfrak{t}$ respectively, see \eqref{sec2_eq_positive-root-algebra} and \eqref{sec2_eq_negative-root-algebra}.
We denote by $\mathfrak{J}=U(\mathfrak{g})\mathfrak{t}^+$ the left ideal of $U(\mathfrak{g})$ generated by $\mathfrak{t}^+$ and consider its normalizer $\Norm \mathfrak{J}$ as a subalgebra of $U(\mathfrak{g})$
\begin{equation}
\Norm \mathfrak{J} := \{ v\in U(\mathfrak{g}) : \mathfrak{J}v\subset \mathfrak{J} \}.
	\end{equation}
The ideal $\mathfrak{J}$ is then a two-sided ideal of $\Norm \mathfrak{J}$ and the Mickelsson algebra is defined as the quotient algebra
\begin{equation}
S(\mathfrak{g},\mathfrak{t}):= (\Norm \mathfrak{J})/\mathfrak{J}.
	\end{equation}
The Mickelsson algebra can be considered as the subalgebra of $\mathfrak{t}$-highest weight vectors in the quotient algebra
\begin{equation}
M(\mathfrak{g},\mathfrak{t}):= U(\mathfrak{g})/\mathfrak{J},
	\end{equation}
in the sense that $S(\mathfrak{g},\mathfrak{t})$ consists exactly of the elements in $M(\mathfrak{g},\mathfrak{t})$ that are annihilated by $\mathfrak{t}^+$.

Letting $R(\mathfrak{h})$ denote the field of fractions of the commutative algebra $U(\mathfrak{h})$ we define the extension of $S(\mathfrak{g},\mathfrak{t})$ by $R(\mathfrak{h})$ to be
\begin{equation}
Z(\mathfrak{g},\mathfrak{t})
	:= S(\mathfrak{g},\mathfrak{t})\otimes_{U(\mathfrak{h})} R(\mathfrak{h}).
	\end{equation} 
This extension is called the Mickelsson-Zhelobenko algebra or extended Mickelsson algebra. We could equivalently have defined it as the quotient algebra $(\Norm \mathfrak{J}')/\mathfrak{J}'$, where $\mathfrak{J}':=U'(\mathfrak{g})\mathfrak{t}^+$ and $U'(\mathfrak{g}):= U(\mathfrak{g})\otimes_{U(\mathfrak{h})} R(\mathfrak{h})$.
Similarly as to how we considered $S(\mathfrak{g},\mathfrak{t})$ as a subalgebra of $M(\mathfrak{g},\mathfrak{t})$, we can consider $Z(\mathfrak{g},\mathfrak{t})$ as a subalgebra of $M'(\mathfrak{g},\mathfrak{t}):=M(\mathfrak{g},\mathfrak{t})\otimes_{U(\mathfrak{h})} R(\mathfrak{h})$. 

To describe the elements and generators of the Mickelsson-Zhelobenko algebra $Z(\mathfrak{g},\mathfrak{t})$ we need the extremal projector of $\mathfrak{t}$.
To define this extremal projector we need to construct an algebra of formal series of monomials in the root vectors of $\mathfrak{t}$.
To this end we let $\Delta^+=\{\alpha_1,\dots,\alpha_m\}$ denote the set of positive roots of $\mathfrak{t}$ and let $E_\alpha$ denote the root vectors of $\alpha\in \Delta^+$. 
We consider for any weight $\mu\in\mathfrak{h}^*$ the vector space $F_\mu(\mathfrak{t})$ over $R(\mathfrak{h})$ of formal series of monomials of weight $\mu$, that is, monomials of the form
\begin{equation}
\label{sec4_eq_weight-monomial-t}
E_{-\alpha_1}^{s_1}\cdots E_{-\alpha_m}^{s_m}E_{\alpha_m}^{t_m}\cdots E_{\alpha_1}^{t_1},
	\end{equation} 
for which $s_1,t_1,\dots,s_m,t_m\in\N_0$ and
\begin{equation}
\mu=(t_1-s_1)\alpha_1+\cdots + (t_m-s_m)\alpha_m.
	\end{equation}
This definition is independent of the choice of ordering of the elements in $\Delta^+$.
We define the space $F(\mathfrak{t})$ to be the direct sum of the spaces $F_\mu(\mathfrak{t})$
\begin{equation}
F(\mathfrak{t}):=\bigoplus_{\mu\in\mathfrak{h}^*} F_\mu(\mathfrak{t}).
	\end{equation}
It can be proven that $F(\mathfrak{t})$ is an associative algebra with respect to multiplication of formal series, see \cite{Zhelobenko-1989}. We equip $F(\mathfrak{t})$ with a Hermitian anti-involution given by $E_\alpha^*:=E_{-\alpha}$ and $E_{ii}^*:=E_{ii}$.

Any element of $F(\mathfrak{t})$ can be considered as an operator acting on $M'(\mathfrak{g},\mathfrak{t})$. That is, 
\begin{equation}
\label{sec4_eq_action-formal-series}
fu:=\sum_{s,t}   E_{-\alpha_1}^{s_1}\cdots E_{-\alpha_m}^{s_m}E_{\alpha_m}^{t_m}\cdots E_{\alpha_1}^{t_1}q_{s,t} u,
	\end{equation}
for any $u\in M'(\mathfrak{g},\mathfrak{t})$ and $f\in F_\mu(\mathfrak{t})$. Here we expressed $f$ as a formal series of monomials \eqref{sec4_eq_weight-monomial-t} with coefficients $q_{s,t}\in R(\mathfrak{h})$. Due to the quotient by $\mathfrak{J}'$ in the definition of $M'(\mathfrak{g},\mathfrak{t})$ all but a finite number of summands in \eqref{sec4_eq_action-formal-series} are congruent to zero modulo $\mathfrak{J}'$. 

We can now define the extremal projector of $\mathfrak{t}$ as an element in $F(\mathfrak{t})$ of weight zero. See \cite{Molev-2002} for details on the extremal projector for $\mathfrak{t}$.

\begin{theo}
The extremal projector $\mathbf{p}\in F(\mathfrak{t})$ is the unique element (up to a factor in $R(\mathfrak{h})$) for which $\mathbf{p}^*=\mathbf{p}$, $\mathbf{p}^2=\mathbf{p}$ and
\begin{equation}
\label{sec4_eq_extremal_projector_identity}
E_\alpha \mathbf{p} = \mathbf{p} E_{-\alpha} = 0, \quad  (\alpha\in \Delta^+).
	\end{equation}
	\end{theo}

The extremal projector for $\mathfrak{t}$ can be constructed explicitly as follows, see \cite{Molev-2002}. 
The positive roots of $\mathfrak{t}$ can be explicitly described as $\alpha_{ij}=\epsilon_i-\epsilon_j$, for $1\leq i<j\leq n$, that is, 
$\Delta^+=\{\alpha_{ij}:1\leq i<j\leq n\}$.
The root vector corresponding to $\alpha_{ij}$ is then $E_{ij}$.
Consider for each root $\alpha_{ij}$ the element $\mathbf{p}_{ij}\in F(\mathfrak{t})$ defined as
\begin{equation}
\label{sec4_eq_partial-extremal-projector}
\mathbf{p}_{ij}
	:= \sum_{k=0}^\infty \frac{(-1)^k}{k!} E_{ji}^k E_{ij}^k \frac{1}{(h_i-h_j+1)_k},
	\end{equation}
where $h_i:= E_{ii}-i+1$, for all $i\in\{1,\dots,n\}$, and where $(x)_k:=x(x+1)\cdots(x+k-1)$ is a Pochhammer symbol.
The extremal projector for $\mathfrak{t}$ is defined as a product 
\begin{equation}
\label{sec4_eq_extremal-projector}
\mathbf{p}:= \prod_{1\leq i<j\leq n} \mathbf{p}_{ij},
	\end{equation}
where the multiplication is determined by a normal order on the roots in $\Delta^+$. The choice of order does not matter as long as it is normal, meaning that for any composite root $\gamma=\alpha+\beta$, with $\alpha,\beta,\gamma\in\Delta^+$, it must hold that $\alpha<\gamma<\beta$ or $\beta<\gamma<\alpha$.
We shall make use of the normal ordering of $\Delta^+$ given by
\begin{equation}
\alpha_{12}<\alpha_{13}<\alpha_{23}<\cdots
	<\alpha_{1n}<\cdots<\alpha_{n-1,n}.
	\end{equation}	
The extremal projector for $\mathfrak{t}$ is then 
\begin{equation}
\mathbf{p}=\mathbf{p}_{12}\mathbf{p}_{13}\mathbf{p}_{23}\cdots\mathbf{p}_{1n}\cdots\mathbf{p}_{n-1,n}.
	\end{equation}	

Following a similar process one can define extemal projectors for many types of Lie (super)algebras, affine Kac-Moody (super)algebras and their quantum analogs, see \cite{Tolstoy-2010}.

The embedding $\mathfrak{t}\subset \mathfrak{g}$ is reductive, so we can decompose $\mathfrak{g}$ under the adjoint representation of $\mathfrak{t}$, that is, $\mathfrak{g}=\mathfrak{t}\oplus V$, where $V$ is the complementary $\mathfrak{t}$-module having a basis consisting of the elements
$B_j^\pm$ and $\{B_i^\pm,B_j^\pm\}$, for $1\leq i\leq j\leq n$.
We let $D$ denote the vector space given by the linear span over $R(\mathfrak{h})$ of the monomials in the basis elements of $V$.
There is then a natural imbedding of $D$ into $M'(\mathfrak{g},\mathfrak{t})$ such that 
\begin{equation}
M'(\mathfrak{g},\mathfrak{t})=D\oplus \mathfrak{t}^-M'(\mathfrak{g},\mathfrak{t}).
	\end{equation}
We let $\delta$ denote the projection from $M'(\mathfrak{g},\mathfrak{t})$ onto $D$ parallel to $\mathfrak{t}^-M'(\mathfrak{g},\mathfrak{t})$.

The following result is a direct translation from the standard Lie algebraic theory of extremal projectors and Mickelsson-Zhelobenko algebras as presented in \cite{Zhelobenko-1989}.
\begin{theo}
\label{sec4_theo_extremal-projector-isomorphism}
Then extremal projector $\mathbf{p}$ is an isomorphism of vector spaces from $D\to Z(\mathfrak{g},\mathfrak{t})$, whose inverse is the restriction of $\delta$ to $Z(\mathfrak{g},\mathfrak{t})$.
	\end{theo}
This result tells us in particular that $Z(\mathfrak{g},\mathfrak{t})$ is spanned by elements of the form $\mathbf{p}v$, for $v\in D$. As we shall see below in Corollary \ref{sec4_coro_mickelsson-zhelobenko-generator}, $Z(\mathfrak{g},\mathfrak{t})$ is in fact generated by the elements
\begin{equation}
\label{sec4_eq_mickelsson-zhelobenko-generator}
\mathbf{p}B_j^\pm
\quad\text{ and }\quad
\mathbf{p}\{B_i^\pm,B_j^\pm\}, \quad (1\leq i\leq j\leq n).
	\end{equation}

We will now give explicit expressions for these elements.
To do so, we need to introduce notation for relevant index sets.
For any $i,j\in\{1,\dots,n\}$ and $s\in\N$ we define the set
\begin{equation}
\mathcal{I}_{ij}(s):=
	\big\{
	I=(i_1,\dots,i_s): i=i_1<i_2<\dots<i_s=j
	\big\},
	\end{equation}
where $\mathcal{I}_{ii}(1)=\{(i)\}$.	
For any $I\in \mathcal{I}_{ij}(s)$ we let $I^\complement$ denote the ordered complement of $I$ in $\{i,\dots,j\}$. That is, $I^\complement:= (i_1',\dots,i_u')$ such that $i_1'<\dots<i_u'$ and $\{i_1',\dots,i_u'\}=\{i,\dots,j\}\backslash \{i_1,\dots,i_s\}$.
Consider for each $I\in \mathcal{I}_{ij}(s)$ the matrix in $M_{nn}(\N_0)$ given by $e_I=\sum_{u=1}^{s-1} e_{i_{u+1},i_u}$. 
The corresponding operator in $U(\mathfrak{t}^-)$, as defined by \eqref{sec2_eq_E-gamma-matrix}, is then
\begin{equation}
E^{e_I}= E_{i_2i_1}\cdots E_{i_si_{s-1}}.
	\end{equation}

\begin{theo}
\label{sec4_theo_mickelsson-zhelobenko-generator}
The elements $\mathbf{p}B_j^\pm\in Z(\mathfrak{g},\mathfrak{t})$, for $1\leq j\leq n$, satisfy 
the following identities modulo $\mathfrak{J}'$.
\begin{equation}
\label{sec4_eq_mickelsson-zhelobenko-generator-1}
\begin{split}
\mathbf{p}B_j^+
	&= 
	\sum_{i=1}^{j}\sum_{s=1}^{j-i+1}\sum_{I\in \mathcal{I}_{ij}(s)} (-1)^{s-1}E^{e_I}B_i^+
	\frac{\prod_{\ell\in I^\complement} (h_\ell-h_j-1)}{\prod_{\ell=i}^{j-1} (h_\ell-h_j)}
,
	\end{split}
	\end{equation}
\begin{equation}
\label{sec4_eq_mickelsson-zhelobenko-generator-2}
\begin{split}
\mathbf{p}B_j^-
	&= 
	\sum_{i=j}^{n}\sum_{s=1}^{i-j+1}\sum_{I\in \mathcal{I}_{ji}(s)} E^{e_I}B_i^-
	\frac{1}{\prod_{\ell\in I, \ell\neq j} (h_j-h_\ell)}
,
	\end{split}
	\end{equation}
Similarly there exists $H_{IJ}^\pm\in R(\mathfrak{h})$ such that the elements $\mathbf{p}\{B_i^\pm,B_j^\pm\}\in Z(\mathfrak{g},\mathfrak{t})$, for $1\leq i\leq j\leq n$, satisfy the following identities modulo $\mathfrak{J}'$.
\begin{equation}
\label{sec4_eq_mickelsson-zhelobenko-generator-3}
\begin{split}
\mathbf{p}\{B_i^+,B_j^+\}
	&=
	\sum_{k=1}^{i}\sum_{s=1}^{i-k+1}\sum_{I\in \mathcal{I}_{ki}(s)}
	\sum_{l=1}^{j}\sum_{t=1}^{j-l+1}\sum_{J\in \mathcal{I}_{lj}(t)} 
	E^{e_I+e_J}\{B_k^+,B_l^+\}H_{IJ}^+
	\end{split}
	\end{equation}
and
\begin{equation}
\label{sec4_eq_mickelsson-zhelobenko-generator-4}
\begin{split}
\mathbf{p}\{B_i^-,B_j^-\}
	&=
	\sum_{k=i}^{n}\sum_{s=1}^{k-i+1}\sum_{I\in \mathcal{I}_{ik}(s)}
	\sum_{l=j}^{n}\sum_{t=1}^{l-j+1}\sum_{J\in \mathcal{I}_{jl}(t)} 
	E^{e_I+e_J}\{B_k^-,B_l^-\}H_{IJ}^-.
	\end{split}
	\end{equation}
	\end{theo}
\begin{proof}
The extremal projector $\mathbf{p}$ is an element of the weight zero subspace $F_0(\mathfrak{t})$ of $F(\mathfrak{t})$. This means that $\mathbf{p}B_j^+$ has weight $\epsilon_j$, is an element of $F_{\epsilon_j}(\mathfrak{t})$ and can be written as an infinite linear combination over $R(\mathfrak{h})$ 
\begin{equation}
\label{sec4_eq_m.z.-generator-proof0}
\mathbf{p}B_j^+ = 
	\sum_{s_1,\dots,s_m\in\N_0}
E_{-\alpha_1}^{s_1}\cdots E_{-\alpha_m}^{s_m}E_{\alpha_m}^{-s_m}\cdots E_{\alpha_1}^{-s_1}B_j^+ H_{s_1,\dots,s_m},
	\end{equation} 
where $H_{s_1,\dots,s_m}\in R(\mathfrak{h})$. The extremal projector $\mathbf{p}$ may be considered as a map from $M'(\mathfrak{g},\mathfrak{t})$ to $Z(\mathfrak{g},\mathfrak{t})$, so by identifying $B_j^+$ with its isomorphism class in $M'(\mathfrak{g},\mathfrak{t})$ we may consider $\mathbf{p}B_j^+$ as an element of $Z(\mathfrak{g},\mathfrak{t})$ which satisfies the identity \eqref{sec4_eq_m.z.-generator-proof0} modulo $\mathfrak{J}'$.
Since $\{\alpha_1,\dots,\alpha_m\}=\Delta^+=\{\epsilon_k-\epsilon_l:1\leq k<l\leq n\}$ and $E_{kl}=E_{\epsilon_k-\epsilon_l}$ we know that, for any $i\in\{1,\dots,m\}$, there exists $k,l\in\{1,\dots,n\}$ with $k<l$ such that $E_{\alpha_i}=E_{kl}$.
The relations \eqref{sec2_eq_relations-osp} tells us that 
\begin{equation}
[E_{kl},B_j^+]= \delta_{jl}B_k^+,
	\end{equation}
for any $k,l\in\{1,\dots,n\}$ with $k<l$.
These observations imply that the sum \eqref{sec4_eq_m.z.-generator-proof0} can be rewritten as follows modulo $\mathfrak{J}'$
\begin{equation}
\label{sec4_eq_m.z.-generator-proof1}
\mathbf{p}B_j^+ = 
	\sum_{i=1}^j
	\sum_{\substack{s_1,\dots,s_m\in\N_0,\\ -s_1\alpha_1-\cdots-s_m\alpha_m+\epsilon_i=\epsilon_j}}
E_{-\alpha_1}^{s_1}\cdots E_{-\alpha_m}^{s_m}B_i^+ \hat{H}_{s_1,\dots,s_m}(i),
	\end{equation} 
where $\hat{H}_{s_1,\dots,s_m}(i)\in R(\mathfrak{h})$. 
Observe now that the identity 
\begin{equation}
\label{sec4_eq_m.z.-generator-proof2}
-s_1\alpha_1-\cdots-s_m\alpha_m+\epsilon_i=\epsilon_j
	\end{equation}
is satisfied if and only if there exists $I\in\mathcal{I}_{ij}(s)$ with $s\in\{1,\dots,j-i+1\}$ such that 
\begin{equation}
E_{-\alpha_1}^{s_1}\cdots E_{-\alpha_m}^{s_m}B_i^+
	=E_{i_2i_1}\cdots E_{i_si_{s-1}}B_i^+=E^{e_I}B_i^+.
	\end{equation}
This means specifically that there exist elements $H_I^+\in R(\mathfrak{h})$ such that \eqref{sec4_eq_m.z.-generator-proof1} can be rewritten as follows modulo $\mathfrak{J}'$
\begin{equation}
\label{sec4_eq_m.z.-generator-proof3}
\mathbf{p}B_j^+
	= 
	\sum_{i=1}^{j}\sum_{s=1}^{j-i+1}\sum_{I\in \mathcal{I}_{ij}(s)} E^{e_I}B_i^+H_I^+.
	\end{equation}
To get \eqref{sec4_eq_mickelsson-zhelobenko-generator-1} we need to determine the elements $H_I^+$. 
By applying a positive root vector $E_{kl}$, with $k<l$, to either side of the identity \eqref{sec4_eq_m.z.-generator-proof3} and using \eqref{sec4_eq_extremal_projector_identity} we get 
\begin{equation}
0	= 
	\sum_{i=1}^{j}\sum_{s=1}^{j-i+1}\sum_{I\in \mathcal{I}_{ij}(s)} E_{kl}E^{e_I}B_i^+H_I^+.
	\end{equation}
By comparing the identities obtained for each positive root vector $E_{kl}$ we get a set of linear equations which can be solved for the coefficients $H_I^+$. The identity \eqref{sec4_eq_mickelsson-zhelobenko-generator-1} is then obtained by substituting these solutions into \eqref{sec4_eq_m.z.-generator-proof3}.
The remaining identities \eqref{sec4_eq_mickelsson-zhelobenko-generator-2}, \eqref{sec4_eq_mickelsson-zhelobenko-generator-3} and \eqref{sec4_eq_mickelsson-zhelobenko-generator-4} are proven in a similar manner.
	\end{proof}

From the identities given in Theorem \ref{sec4_theo_mickelsson-zhelobenko-generator}, we can conclude that the elements $\mathbf{p}B_j^\pm$ and $\mathbf{p}\{B_i^\pm,B_j^\pm\}$ generate $Z(\mathfrak{g},\mathfrak{t})$.

\begin{coro}
\label{sec4_coro_mickelsson-zhelobenko-generator}
The Mickelsson-Zhelobenko algebra $Z(\mathfrak{g},\mathfrak{t})$ is generated by the elements $\mathbf{p}B_j^\pm$ and $\mathbf{p}\{B_i^\pm,B_j^\pm\}$, for $1\leq i\leq j\leq n$. 
	\end{coro}
\begin{proof}
By Theorem \ref{sec4_theo_extremal-projector-isomorphism} any element of $Z(\mathfrak{g},\mathfrak{t})$ is of the form $\mathbf{p}v$, where $v\in D$. Any element $v\in D$ is a linear combination of the monomials in the elements $B_j^\pm$ and $\{B_i^\pm,B_j^\pm\}$. 
By isolating these terms in the identities from Theorem \ref{sec4_theo_mickelsson-zhelobenko-generator} we get recursive relations that upon repeated use will let us expand the elements $B_j^\pm$ and $\{B_i^\pm,B_j^\pm\}$ as linear combinations of elements $\mathbf{p}B_j^\pm$ and $\mathbf{p}\{B_i^\pm,B_j^\pm\}$.
For the elements $B_j^\pm$ these expansions are
\begin{equation}
\label{sec4_eq_mickelsson-zhelobenko-generator-identity-1}
B_j^+=\mathbf{p}B_j^+
	+\sum_{i=1}^{j-1}\sum_{s=2}^{j-i+1}\sum_{I\in \mathcal{I}_{ij}(s)} E^{e_I}\mathbf{p}B_i^+
	\frac{\prod_{\ell\in I^\complement} (h_i-h_\ell+1)}{\prod_{\ell=i+1}^{j} (h_i-h_\ell)}
	\end{equation}
and 
\begin{equation}
\label{sec4_eq_mickelsson-zhelobenko-generator-identity-2}
B_j^-
	=\mathbf{p}B_j^- 
	 + \sum_{i=j+1}^{n}\sum_{s=2}^{i-j+1}\sum_{I\in \mathcal{I}_{ji}(s)} E^{e_I}\mathbf{p}B_i^-
	\frac{1}{\prod_{\ell\in I,\ell\neq i} (h_i-h_\ell)}.
	\end{equation}
Similar expansions can be obtained for the elements $\{B_i^\pm,B_j^\pm\}$.
Using these relations we can write $\mathbf{p}v$ as a linear combinations of monomials in the $\mathbf{p}B_j^\pm$'s and $\mathbf{p}\{B_i^\pm,B_j^\pm\}$'s.
	\end{proof}

For application to the study of $\mathfrak{g}$-modules it is beneficial to consider certain normalizations of the generators given in Corollary \ref{sec4_coro_mickelsson-zhelobenko-generator}. 
We can consider any element $x\in Z(\mathfrak{g},\mathfrak{t})$ as having coefficients in the fraction field $R(\mathfrak{h})$. This means that we can find an element $\pi\in U(\mathfrak{h})$ that acts as a right denominator for $x$ in such a way that $x\pi$ has coefficients in $U(\mathfrak{h})$ instead of in $R(\mathfrak{h})$. It follows that $x\pi$ can be considered an element of the Mickelsson algebra $S(\mathfrak{g},\mathfrak{t})$. 
Using such normalizations we can obtain generators for $S(\mathfrak{g},\mathfrak{t})$ from the generators of $Z(\mathfrak{g},\mathfrak{t})$ given above. Specifically this is done as follows.
\begin{equation}
\label{sec4_eq_mickelsson-zhelobenko-generator-identity-3}
\begin{split}
z_j^+&:= \mathbf{p}B_j^+(h_1-h_j)\cdots(h_{j-1}-h_j),\quad (1\leq j\leq n),\\
z_j^-&:= \mathbf{p}B_j^-(h_j-h_{j+1})\cdots(h_j-h_n),\quad (1\leq j\leq n),\\
z_{ij}^+&:= \mathbf{p}\{B_i^+,B_j^+\}\pi_{ij}^+
\\
z_{ij}^-&:= \mathbf{p}\{B_i^-,B_j^-\}\pi_{ij}^-
,
	\end{split}
	\end{equation}
where $\pi_{ij}^\pm$ are appropriate elements in $U(\mathfrak{h})$ of which the explicit form is not needed in this paper.

The unital associative algebra generated by the elements $z_j^\pm$ and $z_{ij}^\pm$ is thus either $Z(\mathfrak{g},\mathfrak{t})$ or $S(\mathfrak{g},\mathfrak{t})$ depending on whether it is considered as an algebra over $R(\mathfrak{h})$ or $U(\mathfrak{h})$. 
See \eqref{sec5_mickelsson-algebra-generators-explicit-n=3} for explicit calculation of the $z_j^\pm$'s 
in the case $n=3$.

\subsection{Application to lowest weight modules of $\mathfrak{osp}(1|2n)$}
\label{sec4.2}

The most obvious use for the raising and lowering operators $z_j^\pm$ and $z_{ij}^\pm$ is to study the $\mathfrak{t}$-highest weight vectors in $\mathfrak{g}$-modules. This includes branching problems where one wants to describe the structure of simple $\mathfrak{g}$-modules when they are viewed as $\mathfrak{t}$-modules. 
What follows is a collection of general observations regarding the use of raising and lowering operators in the study of $\mathfrak{t}$-highest weight vectors in $\mathfrak{g}$-modules.

Consider a simple lowest weight module $L$ of $\mathfrak{g}$ whose lowest weight vector is $w_0$. Since $\mathfrak{J}w_0=\{0\}$ it follows that the algebra $M(\mathfrak{g},\mathfrak{t})$ has a natural action on $L$ such that 
\begin{equation}
M(\mathfrak{g},\mathfrak{t})w_0=U(\mathfrak{g})w_0=L.
	\end{equation}
We can additionally describe the space $L^+$ spanned by all $\mathfrak{t}$-highest weight vectors in $L$ as follows 
\begin{equation}
L^+	:=\{ v\in L : \mathfrak{t}^+v=0\} 
	= S(\mathfrak{g},\mathfrak{t})w_0
	\end{equation}
This means that $S(\mathfrak{g},\mathfrak{t})$ has a natural action on $L^+$ and that, for any $x\in S(\mathfrak{g},\mathfrak{t})$, there exists an element $\hat{x}\in U(\mathfrak{g})$ that is unique modulo $\mathfrak{J}$ and for which $xv=\hat{x}v$, for all $v\in L^+$.

Let us denote by $L_\mu^+$ the weight subspace of $L^+$ corresponding to the weight $\mu\in\mathfrak{h}^*$, that is,
\begin{equation}
L_\mu^+:=\big\{ v\in L^+ : E_{ii}v=\mu_iv,\ i\in\{1,\dots,n\}\big\}.
	\end{equation} 
This lets us make the following observation regarding the generators of $S(\mathfrak{g},\mathfrak{t})$.
\begin{lemm}
\label{sec4_lemm_raising-lowering-action-weight-spaces}
For any $1\leq i\leq j\leq n$ and $\mu\in \mathfrak{h}^*$ we have
\begin{equation}
z_j^\pm L_\mu^+\subset L_{\mu\pm\epsilon_j}^+
\quad\text{ and }\quad	
z_{ij}^\pm L_\mu^+\subset L_{\mu\pm\epsilon_i\pm\epsilon_j}^+.
	\end{equation}
	\end{lemm}

To apply these ideas in the context discussed in the rest of this paper we put $L=L(p)$. Proposition \ref{sec3_prop_highest-weight-vector} tells us that $L(p)^+$ has a basis consisting
of the vectors $\Omega_\lambda$, for $\lambda\in\mathbb{P}$ with $\ell(\lambda)\leq p$. In Corollary \ref{sec4_coro_action-raising-lowering-h.w.-vector} we describe explicitly the actions of the raising and lowering operators $z_j^\pm$, for $j\in\{1,\dots,n\}$, on the $\Omega_\lambda$'s and in \eqref{sec4_eq_gl-h.w.-vector-raising-operators} we express the $\Omega_\lambda$'s as monomials in the raising operators $z_j^+$, for $j\in\{1,\dots,n\}$, acting on the $\mathfrak{osp}(1|2n)$-lowest weight vector $v_0$ in $L(p)$.

\subsection{The action of $B_j^\pm$ on $E^\gamma\Omega_\lambda$}
\label{sec4.3}

We will now take a closer look at the problem of finding formulas for the actions of the operators $B_j^\pm$ on the vectors in $L(p)$ of the form $E^{\gamma}\Omega_\lambda$. We are able to express the resulting vectors $B_j^\pm E^{\gamma}\Omega_\lambda$ as linear combinations of vectors of the form $E^{\gamma'}\Omega_\mu$. However, such expansions are not in general expansions into linear combinations of the vectors from the basis $\{E^{\gamma_A}\Omega_A\}$ for $L(p)$. The problem of obtaining such basis expansions is discussed at the end of this section and is solved in Proposition~\ref{sec5_prop_action-h.w.-n=3} for the case $n=3$. 

In the following result we use the expressions for the raising and lowering operators $\mathbf{p}B_j^\pm$ given in Theorem \ref{sec4_theo_mickelsson-zhelobenko-generator} to get expansions of the vectors $B_j^\pm\Omega_\lambda$.

\begin{prop}
\label{sec4_prop_h.w.-action}
Let $\lambda\in \mathbb{P}$ and $j\in\{1,\dots,n\}$. Then
\begin{equation}
\label{sec4_eq_h.w.-action-1}
B_j^+\Omega_\lambda 
	= 
	\sum_{i=1}^{j}\sum_{s=1}^{j-i+1}
	\sum_{I\in\mathcal{I}_{ij}(s)}
	d_{i}^+(\lambda)
	\frac{\prod_{\ell\in I^\complement} (\lambda_i-\lambda_\ell-i+\ell+1)}{\prod_{\ell=i+1}^{j} (\lambda_i-\lambda_\ell-i+\ell)}
	E^{e_I}
	\Omega_{\lambda+\epsilon_i}
	\end{equation}
and
\begin{equation}
\label{sec4_eq_h.w.-action-2}
B_j^-\Omega_\lambda 
	= 
	\sum_{i=j}^{n}\sum_{s=1}^{i-j+1}
	\sum_{I\in\mathcal{I}_{ji}(s)}
	d_{i}^-(\lambda)
	\frac{1}{\prod_{\ell\in I,\ell\neq i} (\lambda_i-\lambda_\ell-i+\ell)}
	E^{e_I}\Omega_{\lambda-\epsilon_i},
	\end{equation}
where $\Omega_{\lambda\pm\epsilon_i}:=0$ if $\lambda\pm\epsilon_i\notin\mathbb{P}$,
\begin{equation}
d_{i}^+(\lambda)
	=
	\frac{(-1)^{\sum_{\alpha=i}^n(\alpha+1)(\lambda_\alpha-\lambda_{\alpha+1}+\delta_{\alpha i})}}{(\lambda_i+1)i}
	\bigg(
	\prod_{\ell=1}^{i-1}\frac{\lambda_i-\lambda_\ell-i+\ell+1}{\lambda_i-\lambda_\ell-i+\ell+[\lambda_i-\lambda_\ell]_2}
	\bigg)
	\end{equation}
and
\begin{equation}
d_{i}^-(\lambda)
	=
	\frac{(\lambda_i+1)i(\lambda_i+n-i+[\lambda_i]_2(p-n))}{(-1)^{\sum_{\alpha=i}^n(\alpha+1)(\lambda_\alpha-\lambda_{\alpha+1})}}
	\bigg(
	\prod_{\ell=i+1}^n\frac{\lambda_i-\lambda_\ell-i+\ell-1}{\lambda_i-\lambda_\ell-i+\ell-[\lambda_i-\lambda_\ell]_2}
	\bigg).
	\end{equation}
	\end{prop}
\begin{proof}
Recall that the weight of $\Omega_\lambda$ is $\lambda+\frac{p}{2}$ and that $\lambda$ is a partition. Because of this one does not run into the problem of dividing by zero when acting on $\Omega_\lambda$ with operators of the form $\frac{1}{h_k-h_\ell}$, for $k\neq \ell$, which appear in the expressions \eqref{sec4_eq_mickelsson-zhelobenko-generator-1} and \eqref{sec4_eq_mickelsson-zhelobenko-generator-2} of $\mathbf{p}B_j^\pm$.
These observations imply  that the action of $\mathbf{p}B_j^\pm$ on $\Omega_\lambda$ is well defined.
%
Together Proposition \ref{sec3_prop_highest-weight-vector} and Lemma \ref{sec4_lemm_raising-lowering-action-weight-spaces} imply that there exist coefficients $d_{i}^\pm(\lambda)\in\C$ such that  
\begin{equation}
\label{sec4_eq_Zhelobenko-raising-lowering-action-Omega-lambda}
(\mathbf{p}B_i^\pm)\Omega_\lambda= 
\begin{cases}
d_{i}^\pm(\lambda)\Omega_{\lambda\pm\epsilon_i}, & \text{ if } \lambda\pm \epsilon_i\in \mathbb{P} \text{ and } \ell(\lambda \pm \epsilon_i)\leq p,\\
0, \text{ otherwise},
	\end{cases}
	\end{equation}
for any $i\in\{1,\dots,n\}$. 
Using the identities \eqref{sec4_eq_mickelsson-zhelobenko-generator-identity-1} and \eqref{sec4_eq_mickelsson-zhelobenko-generator-identity-2} together with \eqref{sec4_eq_Zhelobenko-raising-lowering-action-Omega-lambda} we immediately get the expansions \eqref{sec4_eq_h.w.-action-1} and \eqref{sec4_eq_h.w.-action-2}. The coefficients $d_i^\pm(\lambda)$ are calculated in Appendix \ref{appA}.
	\end{proof}
%
	
As a consequence of observations made in the proof of Proposition \ref{sec4_prop_h.w.-action} we get the following corollary:
\begin{coro}
\label{sec4_coro_action-raising-lowering-h.w.-vector}
Let $j\in\{1,\dots,n\}$ and $\lambda\in\mathbb{P}$ with $\ell(\lambda)\leq p$, then 
\begin{equation}
z_j^+\Omega_\lambda= 
\begin{cases}
d_{j}^+(\lambda)\prod_{\ell=1}^{j-1}(\lambda_\ell-\lambda_j-\ell+j)\Omega_{\lambda+\epsilon_j}, & \text{ if } \lambda+ \epsilon_j\in \mathbb{P} \text{ and } \ell(\lambda + \epsilon_j)\leq p,\\
0, \text{ otherwise},
	\end{cases}
	\end{equation}
and
\begin{equation}
z_j^-\Omega_\lambda= 
\begin{cases}
d_{j}^-(\lambda)\prod_{\ell=j+1}^{n}(\lambda_j-\lambda_\ell-j+\ell)\Omega_{\lambda-\epsilon_j}, & \text{ if } \lambda- \epsilon_j\in \mathbb{P} \text{ and } \ell(\lambda - \epsilon_j)\leq p,\\
0, \text{ otherwise}.
	\end{cases}
	\end{equation}
	\end{coro}

In order to give formulas for the actions of $B_j^\pm$ on a vector $E^\gamma\Omega_\lambda$ in $L(p)$ we will need the following two technical lemmas in addition to Proposition \ref{sec4_prop_h.w.-action}. We forego giving explicit proofs of these lemmas since they result from simple yet tedious calculations using the relations \eqref{sec2_eq_relations-osp} and \eqref{sec2_eq_relations-gl}.

\begin{lemm}
\label{sec4_lemm_action_step1}
Let $\ell\in\{1,\dots,n\}$, $\gamma\in M_{nn}(\N_0)$ and $\lambda\in \mathbb{P}$ with $\ell(\lambda)\leq p$, then we have
\begin{equation}
B_\ell^+ E^\gamma\Omega_\lambda
	=
	\sum_{j=\ell}^n\sum_{t=1}^{j-\ell+1}\sum_{J\in\mathcal{I}_{\ell j}(t)}
	(-1)^{t-1}\prod_{u=1}^{t-1} \gamma_{j_{u+1}j_u}
	E^{\gamma - e_J}
	B_j^+\Omega_\lambda,
	\end{equation}
and
\begin{equation}
B_\ell^- E^\gamma\Omega_\lambda 
	=
	\sum_{j=1}^\ell  \gamma_{\ell j}^{1-\delta_{\ell j}}
	E^{\gamma - e_{\ell j}}
	B_j^-\Omega_\lambda
	\end{equation}
	\end{lemm}

\begin{lemm}
\label{sec4_lemm_action_step3}
Let $\gamma\in M_{nn}(\N_0)$ and $I\in \mathcal{I}_{ij}(s)$. Then we have
\begin{equation}
E^\gamma E^{e_I}
	=
	\sum_{v_2=i_2}^n\cdots \sum_{v_t=i_s}^n
	\prod_{u=2}^s
	\gamma_{v_ui_u}^{(1-\delta_{v_ui_u})}
	E^{\gamma+\sum_{u=2}^s e_{v_ui_{u-1}}-e_{v_ui_u}}.
	\end{equation}
	\end{lemm}

Given $\ell\in\{1,\dots,n\}$, $\gamma\in M_{nn}(\N_0)$ and $\lambda\in \mathbb{P}$ with $\ell(\lambda)\leq p$ we can use Lemma \ref{sec4_lemm_action_step1}, Proposition \ref{sec4_prop_h.w.-action} and  Lemma \ref{sec4_lemm_action_step3}, in that order, to obtain the following expansions.
\begin{equation}
\label{sec4_eq_action-b-plus}
\begin{split}
B_\ell^+ E^\gamma\Omega_\lambda
	=
\sum_{J\in\mathcal{I}_{\ell j}(t)}
&
\sum_{I\in\mathcal{I}_{ij}(s)}
\sum_{v_2=i_2}^n\cdots \sum_{v_t=i_s}^n
\left(
(-1)^{t-1}d_j^+(\lambda)\prod_{u=1}^{t-1} \gamma_{j_{u+1}j_u}\prod_{u=2}^s(\gamma-e_J)_{v_ui_u}^{1-\delta_{v_ui_u}}
\right)
	\times\\&\times	
	E^{\gamma-e_J+\sum_{u=2}^s e_{v_ui_{u-1}}-e_{v_ui_u}}\Omega_{\lambda+\epsilon_i},
	\end{split}
	\end{equation}
where we are implicitly also summing over $j\in\{\ell,\dots,n\}$, $t\in\{1,\dots,j-\ell+1\}$, $i\in\{1,\dots,j\}$ and $s\in\{1,\dots,j-i+1\}$, and
\begin{equation}
\label{sec4_eq_action-b-minus}
\begin{split}
B_\ell^- E^\gamma\Omega_\lambda
	=
\sum_{j=1}^\ell
&
\sum_{I\in\mathcal{I}_{ji}(s)}
\sum_{v_2=i_2}^n\cdots \sum_{v_t=i_s}^n
\left(
\gamma_{\ell j}^{1-\delta_{\ell j}}d_j^-(\lambda)\prod_{u=2}^s(\gamma-e_{\ell j})_{v_ui_u}^{1-\delta_{v_ui_u}}
\right)
	\times\\&\times	
	E^{\gamma-e_{\ell j}+\sum_{u=2}^s e_{v_ui_{u-1}}-e_{v_ui_u}}\Omega_{\lambda-\epsilon_i},
	\end{split}
	\end{equation}
where we are implicitly also summing over $i\in\{j,\dots,n\}$ and $s\in\{1,\dots,i-j+1\}$.
To discuss the implications of these expansions we need to take a closer look at the vectors 
\begin{equation}
E^{\gamma-e_J+\sum_{u=2}^s e_{v_ui_{u-1}}-e_{v_ui_u}}\Omega_{\lambda+\epsilon_i}
\quad\text{ and }\quad
E^{\gamma-e_{\ell j}+\sum_{u=2}^s e_{v_ui_{u-1}}-e_{v_ui_u}}\Omega_{\lambda-\epsilon_i}.
	\end{equation}
To determine whether or not these vectors are basis vectors, we need to analyse the exponent matrices using Lemma \ref{sec3_lemm_betweenness-like-condition}. To this end we introduce the following shorthand notation
\begin{equation}
\gamma^+:=\gamma-e_J+\sum_{u=2}^s e_{v_ui_{u-1}}-e_{v_ui_u}
	\end{equation}
and
\begin{equation}
\gamma^-:=\gamma-e_{\ell j}+\sum_{u=2}^s e_{v_ui_{u-1}}-e_{v_ui_u}.
	\end{equation}
We can modify these matrices to get matrices $\bar{\gamma}^\pm$ that are lower triangular and have column sums $\lambda\pm \epsilon_i$ as follows
\begin{equation}
\bar{\gamma}^\pm_{kl}:=
\begin{cases}
\gamma^\pm_{kl},& \text{ if } k>l,\\
\lambda_l\pm\delta_{il}-\sum_{m=1}^{l-1}\gamma^\pm_{ml}, & \text{ if } k=l,\\
0 & \text{ if } k<l,
	\end{cases}
	\end{equation}
for all $k,l\in\{1,\dots,n\}$. If $\bar{\gamma}^\pm$ has a negative diagonal entry, then $E^{\gamma^\pm}\Omega_{\lambda\pm\epsilon_i}=0$, if not, then $\bar{\gamma}^\pm$ is a lower triangular matrix in $M_{nn}(\N_0)$ with column sum $\lambda\pm\epsilon_i$. Using \eqref{sec2_eq_E-gamma-matrix} and Theorem \ref{sec3_theo_PBW-type-basis-vector-variable-case} we get
\begin{equation}
E^{\gamma^\pm}\Omega_{\lambda\pm\epsilon_i}
	=E^{\bar{\gamma}^\pm}\Omega_{\lambda\pm\epsilon_i}
	=\frac{(\lambda\pm\epsilon_i)!}{\diag(\bar{\gamma}^\pm)!}\Omega_{D(\bar{\gamma}^\pm)},
	\end{equation}
where $D(\bar{\gamma}^\pm)\in\mathbb{E}(\lambda\pm\epsilon_i)$.
To determine whether or not $E^{\bar{\gamma}^\pm}\Omega_{\lambda\pm\epsilon_i}$ is an element of the basis $\{E^{\gamma_A}\Omega_{\lambda_A}\}$ for $L(p)$, we need to check if $D(\bar{\gamma}^\pm)$ is a s.s. Young tableau in $\mathbb{Y}(\lambda\pm\epsilon_i)$ and $\ell(\lambda\pm\epsilon_i)\leq p$. The necessary and sufficient conditions for this are given in Lemma \ref{sec3_lemm_betweenness-like-condition}.

It is unfortunately not generally the case that $D(\bar{\gamma}^\pm)$ satisfies the conditions to be a s.s. Young tableau in $\mathbb{Y}(\lambda\pm\epsilon_i)$, even if $\gamma=\gamma_A$, for some $A\in \mathbb{Y}(\lambda)$.
If this was the case then the identities \eqref{sec4_eq_action-b-plus} and \eqref{sec4_eq_action-b-minus} would completely describe matrix elements of the action of $\mathfrak{osp}(1|2n)$ on the basis $\{E^{\gamma_A}\Omega_{\lambda_A}\}$ for $L(p)$.
Since this is not the case a different strategy is needed in order to obtain these matrix elements from \eqref{sec4_eq_action-b-plus} and \eqref{sec4_eq_action-b-minus}.
One possible strategy would be to work out how the vectors $E^{\bar{\gamma}^\pm}\Omega_{\lambda\pm\epsilon_i}$ expand into linear combinations of the vectors $E^{\gamma_B}\Omega_{\lambda\pm\epsilon_i}$, for $B\in\mathbb{Y}(\lambda\pm\epsilon_i)$, which form the PBW-type basis for the $\mathfrak{gl}(n)$-module $V(\lambda\pm\epsilon_i+\frac{p}{2})$ defined in Theorem \ref{sec2_theo_PBW-type-basis-u(n)}. Knowing these expansions is essentially equivalent to knowing the matrix elements of the action of $\mathfrak{gl}(n)$ on the PBW-type basis for $V(\lambda\pm\epsilon_i+\frac{p}{2})$. Such knowledge is unfortunately missing from the literature. In fact, the only basis for $V(\lambda\pm\epsilon_i+\frac{p}{2})$ for which matrix elements are known is a GZ-basis, see \cite{Molev-2002}.

In conclusion, the identities \eqref{sec4_eq_action-b-plus} and \eqref{sec4_eq_action-b-minus} do not give us the matrix elements of the action of $\mathfrak{osp}(1|2n)$ on the basis $\{E^{\gamma_A}\Omega_{\lambda_A}\}$ for $L(p)$. They do however connect this problem with the problem of finding such matrix elements for the action of $\mathfrak{gl}(n)$ on the PBW-type basis for $V(\lambda\pm\epsilon_i+\frac{p}{2})$.


\subsection{Raising operators and the Gel'fand-Zetlin-basis for $L(p)$}
\label{sec4.4}

The first basis that was obtained for the $\mathfrak{osp}(1|2n)$-module $L(p)$ was a GZ-basis related to the branching 
\begin{equation}
\label{sec4_eq_branching-chain}
\mathfrak{osp}(1|2n)\supset\mathfrak{gl}(n)\supset\cdots\supset \mathfrak{gl}(1).
	\end{equation}
This is the only basis for $L(p)$ for which the matrix elements of the action of $\mathfrak{osp}(1|2n)$ on the basis are known explicitly, see \cite{Lievens-Stoilova-VanderJeugt-2008}. The main disadvantage of this basis is that no expression is known for the basis elements in terms of operators from $U(\mathfrak{g})$ acting on the lowest weight vector of $L(p)$.
In this section we will use raising operators to obtain such expressions.
In Section \ref{sec5} we will express the vectors of the GZ-basis as polynomials in the $B_j^+$'s acting on $v_0$ for the case $n=3$. This is done using Proposition~\ref{sec5_prop_transition-matrix-n=3} which explicitly expands the vectors of the GZ-basis as linear combinations of the vector from the basis $\{E^{\gamma_A}\Omega_{\lambda_A}\}$.

A GZ-basis for $L(p)$ is a union of GZ-bases for the highest weight $\mathfrak{gl}(n)$-modules $V(\lambda+\frac{p}{2})$, for $\lambda\in\mathbb{P}$ with $\ell(\lambda)\leq p$, each respecting the branching $\mathfrak{gl}(n)\supset\cdots\supset \mathfrak{gl}(1)$. 
In \cite{Molev-2002} such GZ-bases for the $\mathfrak{gl}(n)$-modules $V(\lambda+\frac{p}{2})$ are constructed with the basis vectors being expressed using lowering operators from the Mickelsson-Zhelobenko algebras $Z(\mathfrak{gl}(m),\mathfrak{gl}(m-1))$, for $m\in\{2,\dots,n\}$. We briefly sketch this construction.

The extremal projector of $\mathfrak{gl}(m)$ is given by the formulas \eqref{sec4_eq_partial-extremal-projector} and \eqref{sec4_eq_extremal-projector} where the index $n$ is exchanged for $m-1$. The Mickelsson-Zhelobenko algebra $Z(\mathfrak{gl}(m),\mathfrak{gl}(m-1))$ is then generated by the raising operators 
\begin{equation}
\label{sec4_eq_gl-raising}
\mathbf{p}E_{jm}
	=
	E_{jm}
	+
	\sum_{i=1}^{j-1}\sum_{s=2}^{j-i+1}\sum_{I\in\mathcal{I}_{ij}(s)}
	E_{i_si_{s-1}}\cdots E_{i_2i_1}E_{i_1 m}
	\frac{1}{\prod_{\ell\in I,\ell\neq j}(h_j-h_\ell)}
	\end{equation}
and lowering operators 
\begin{equation}
\label{sec4_eq_gl-lowering}
\mathbf{p}E_{mj}
	=
	E_{mj}
	+
	\sum_{i=j+1}^{m-1}\sum_{s=2}^{i-j+1}\sum_{I\in\mathcal{I}_{ji}(s)}
	E_{i_2i_1}\cdots E_{i_si_{s-1}}E_{m i_s}
	\frac{1}{\prod_{\ell\in I,\ell\neq j}(h_j-h_\ell)},
	\end{equation}
for $j\in\{1,\dots,m-1\}$.
Similarly to the identities given in Theorem \ref{sec4_theo_mickelsson-zhelobenko-generator} these identities hold only modulo the ideal in $Z(\mathfrak{gl}(m),\mathfrak{gl}(m-1))$ corresponding to the ideal $\mathfrak{J}'$ in $Z(\mathfrak{osp}(1|2n),\mathfrak{gl}(n))$. For more details, see \cite{Molev-2002}.

By multiplications with the right denominators we can obtain raising and lowering operators in the Mickelsson algebra $S(\mathfrak{gl}(m),\mathfrak{gl}(m-1))$:
\begin{equation}
\label{sec4_eq_gl-raising-lowering}
\begin{split}
y_{jm}&=\mathbf{p}E_{jm}(h_j-h_1)\cdots(h_j-h_{j-1})\\
y_{mj}&=\mathbf{p}E_{mj}(h_j-h_{j+1})\cdots(h_j-h_{m-1}),
	\end{split}
	\end{equation}
for $j\in\{1,\dots,m-1\}$.
The elements of the GZ-basis for $V(\lambda+\frac{p}{2})$ constructed in \cite{Molev-2002} can be indexed by s.s.\ Young tableaux $A\in\mathbb{Y}(\lambda)$ with $\ell(\lambda_A)\leq p$. The basis element corresponding to $A$ is then 
\begin{equation}
\label{sec4_eq_GZ-basis-gl}
	\prod_{k=2,\dots,n}^\rightarrow 
	\left(y_{k1}^{(\gamma_A)_{k1}}\cdots y_{k,k-1}^{(\gamma_A)_{k,k-1}}\right)
	\xi_\lambda,
	\end{equation}
where $\xi_\lambda$ is the highest weight vector of $V(\lambda+\frac{p}{2})$. The arrow refers to the ordering of the product ($k=2$ is leftmost and $k=n$ is rightmost).
We can now construct a GZ-basis for $L(p)$.

\begin{theo}
\label{sec4_theo_GZ-basis}
Let $v_0$ denote the lowest weight vector of $L(p)$. Then $L(p)$ has a GZ-basis consisting of the vectors 
\begin{equation}
v_A:= \prod_{k=2,\dots,n}^\rightarrow 
	\left(y_{k1}^{(\gamma_A)_{k1}}\cdots y_{k,k-1}^{(\gamma_A)_{k,k-1}}\right)
	(z_n^+)^{\lambda_n}\cdots (z_1^+)^{\lambda_1}v_0,
	\end{equation}
for all $A\in\mathbb{Y}(\lambda)$ and $\lambda\in\mathbb{P}$ with $\ell(\lambda)\leq p$. Furthermore, the vector $(z_n^+)^{\lambda_n}\cdots (z_1^+)^{\lambda_1}v_0$ is a $\mathfrak{gl}(n)$-highest weight vector of weight $\lambda+\frac{p}{2}$.
	\end{theo}
Using Corollary \ref{sec4_coro_action-raising-lowering-h.w.-vector} we can state the relationship between $(z_n^+)^{\lambda_n}\cdots (z_1^+)^{\lambda_1}v_0$ and $\Omega_\lambda$ as follows.
\begin{equation}
\label{sec4_eq_gl-h.w.-vector-raising-operators}
\begin{split}
(z_n^+)^{\lambda_n}\cdots(z_1^+)^{\lambda_1}v_0
	&=
	\left(\prod_{j=1}^n\frac{(-1)^{\frac{\lambda_j(\lambda_j+1)}{2}}}{\lambda_j!j^{\lambda_j}}\prod_{k=0}^{\lambda_j-1}\prod_{\ell=1}^{j-1}(k-\lambda_\ell-j+\ell+1-[k-\lambda_\ell ]_2)\right)\Omega_\lambda.
	\end{split}
	\end{equation}
The GZ-basis for $L(p)$ constructed in Theorem \ref{sec4_theo_GZ-basis} differs from the one obtained in \cite{Lievens-Stoilova-VanderJeugt-2008} by a scalar factor on each basis vector. These scalars can be obtained by making $L(p)$ a Hilbert space. This is done using the Hermitian inner product $\langle \cdot,\cdot\rangle$ defined by the relations 
\begin{equation}
\langle v_0,v_0\rangle =1 
\quad \text{ and } \quad 
\langle B_j^\pm v_A, v_B\rangle = \langle v_A, B_j^\mp v_B\rangle,
	\end{equation}
for all $j\in\{1,\dots,n\}$ and $A,B\in\mathbb{Y}(\lambda)$ with $\lambda\in\mathbb{P}$ and $\ell(\lambda)\leq p$. 
The GZ-basis constructed in \cite{Lievens-Stoilova-VanderJeugt-2008} is then an orthonormal basis for $L(p)$ with respect to $\langle \cdot,\cdot\rangle$ and the elements of this basis are given by 
\begin{equation}
\label{sec4_eq_normalization-GZ-basis}
|m_A):=\frac{1}{\sqrt{\langle v_A,v_A\rangle}}v_A,
	\end{equation}
for all $A\in\mathbb{Y}(\lambda)$ with $\lambda\in\mathbb{P}$ and $\ell(\lambda)\leq p$. Of course these observations mean that the GZ-basis obtained in Theorem \ref{sec4_theo_GZ-basis} is orthogonal with respect to $\langle \cdot,\cdot\rangle$.
In \cite{Lievens-Stoilova-VanderJeugt-2008} the basis elements $|m_A)$ are parametrized by GZ-patterns and not s.s. Young tableaux. These parametrizations are equivalent and the entries $m_{ij}$ of the GZ-pattern corresponding to $|m_A)$ can be calculated as follows using the exponent matrix $\gamma_A$.
\begin{equation}
m_{ij}=\sum_{k=i}^j (\gamma_A)_{ki}, \quad (1\leq i\leq j\leq n).
	\end{equation}

We end this section by presenting a result that relate the basis $\{E^{\gamma_A}\Omega_{\lambda_A}\}$ with the GZ-basis constructed in Theorem \ref{sec4_theo_GZ-basis}.
This result is an extrapolation of Theorem A from \cite{Molev-Yakimova-2018} which relates the PBW-type basis for $V(\lambda+\frac{p}{2})$ with the GZ-basis for $V(\lambda+\frac{p}{2})$. 
Recall that these bases are given in Theorem \ref{sec2_theo_PBW-type-basis-u(n)} and \eqref{sec4_eq_GZ-basis-gl} respectively.

\begin{theo}
\label{sec4_theo_triangular-transition-matrix}
There exists a total order $>$ on the index set $\{A\in \mathbb{Y}(\lambda): \lambda\in\mathbb{P},\ell(\lambda)\leq p\}$, such that 
\begin{equation}
v_A = \sum_{B\geq A} T_{AB} E^{\gamma_B}\Omega_{\lambda_B},
	\end{equation}
$T_{AB}\in\C$ and $T_{AA}\neq 0$, for all $A\in\mathbb{Y}(\lambda)$ and $\lambda\in\mathbb{P}$ with $\ell(\lambda)\leq p$.
	\end{theo}

Theorem \ref{sec4_theo_triangular-transition-matrix} tells us that there is a upper triangular transition matrix between the two bases. Details regarding the ordering $<$ of $\{A\in \mathbb{Y}(\lambda): \lambda\in\mathbb{P},\ell(\lambda)\leq p\}$ can be found in \cite{Molev-Yakimova-2018}. The coefficients $T_{AB}$ are calculated explicitly in Proposition~\ref{sec5_prop_transition-matrix-n=3} for $n=3$.

Together Theorem \ref{sec4_theo_triangular-transition-matrix}, Corollary \ref{sec3_coro_PBW-bases-Omega_A} and \eqref{sec3_eq_def-Omega_A} present a way to express any GZ-basis vector $v_A$ as a polynomial in the $B_j^+$'s acting on the $\mathfrak{osp}(1|2n)$-lowest weight vector $v_0$. Recall here that the $B_j^+$'s represent parabosonic creation operators in the context of parabosonic field theories.

\section{Example: The case $n=3$}
\label{sec5}

We will now study the case $n=3$ as an illustrative example of the results obtained in this paper. In this section we will emphasize the connections with the theory of parabosons which has been noted at various points throughout the paper. From this perspective $\mathfrak{osp}(1|6)$ is the Lie superalgebra generated by the parabosonic creation and annihilation operators $B_1^+$, $B_2^+$, $B_3^+$, $B_1^-$, $B_2^-$ and $B_3^-$, $L(p)$ is the Fock space of $3$ parabosonic particles of order~$p$ and $v_0$ is the vacuum state in $L(p)$.

In Sections \ref{sec2} and \ref{sec3} we constructed a new basis for the Fock space $L(p)$ consisting of the vectors 
\begin{equation}
\label{sec5_eq_basis-element-identity}
E^{\gamma_A} \Omega_{\lambda_A}=\frac{\lambda_A!}{\diag(\gamma_A)!}\Omega_A
	\end{equation}
for all $A\in\mathbb{Y}$ with $\ell(\lambda_A)\leq p$, that is, for all s.s. Young tableaux with at most $p$ rows. When $n=3$ and $p\geq 3$ any such tableau $A$ takes either of the follows two forms
\begin{equation}
\ytableausetup{aligntableaux=top,smalltableaux}
A=\underbrace{\ytableaushort{1,2,3}\cdots\ytableaushort{1,2,3}}_{k_{123}}
	\underbrace{\ytableaushort{1,2}\cdots\ytableaushort{1,2}}_{k_{12}}
	\underbrace{\ytableaushort{1,3}\cdots\ytableaushort{1,3}}_{k_{13}}
	\underbrace{\ytableaushort{1}\cdots\ytableaushort{1}}_{k_{1}}
	\underbrace{\ytableaushort{2}\cdots\ytableaushort{2}}_{k_{2}}
	\underbrace{\ytableaushort{3}\cdots\ytableaushort{3}}_{k_{3}}
	\end{equation}
or 
\begin{equation}
\ytableausetup{aligntableaux=top,smalltableaux}
A=\underbrace{\ytableaushort{1,2,3}\cdots\ytableaushort{1,2,3}}_{k_{123}}
	\underbrace{\ytableaushort{1,2}\cdots\ytableaushort{1,2}}_{k_{12}}
	\underbrace{\ytableaushort{1,3}\cdots\ytableaushort{1,3}}_{k_{13}}
	\underbrace{\ytableaushort{2,3}\cdots\ytableaushort{2,3}}_{k_{23}}
	\underbrace{\ytableaushort{2}\cdots\ytableaushort{2}}_{k_{2}}
	\underbrace{\ytableaushort{3}\cdots\ytableaushort{3}}_{k_{3}},
	\end{equation}
where $k_{i_1,\dots,i_s}\in \N_0$, for $s\in\{1,2,3\}$ and $1\leq i_1<\cdots<i_s\leq 3$, is the number of columns in $A$ with entries $i_1,\dots,i_s$.

The description of the basis elements of $L(p)$ as given by the identity  \eqref{sec5_eq_basis-element-identity} comes from Corollary \ref{sec3_coro_PBW-bases-Omega_A}. 
The first expression $E^{\gamma_A} \Omega_{\lambda_A}$ illustrates the relationship of the basis with the branching $\mathfrak{osp}(1|6)\supset \mathfrak{gl}(3)$. Here $\Omega_\lambda$ is a $\mathfrak{gl}(3)$-highest weight vector in  $L(p)$, $E^{\gamma_A}$ is a monomial in the $\mathfrak{gl}(3)$-positive root vectors $E_{ij}$, for $1\leq j<i\leq 3$, and $\gamma_A$ is the exponent matrix of $A$ which has entries
\begin{equation}
(\gamma_A)_{ij}=\#\{ \text{$i$'s in the $j$'th row of $A$} \}.
	\end{equation}
Using the shorthand $\lambda=\lambda_A$ and $\gamma=\gamma_A$ we can write the vector $E^{\gamma_A} \Omega_{\lambda_A}$ in terms of the paraboson creation and annihilation operators using the identities $\Omega_\lambda=\lambda!\omega_\lambda$ and $E_{ij}=\frac{1}{2}\{B_i^+,B_j^-\}$ together with \eqref{sec2_eq_E-gamma-matrix}, \eqref{sec3_eq_def-omega_A} and \eqref{sec3_eq_def-multi-bracket}
\begin{equation}
\label{sec5_eq_basis-vector-expression-1}
\begin{split}
E^{\gamma} \Omega_{\lambda}
	= 
	E_{21}^{\gamma_{21}}E_{31}^{\gamma_{31}}E_{32}^{\gamma_{32}}\Omega_\lambda
	&=
	\frac{\lambda_1!\lambda_2!\lambda_3!}{2^{\gamma_{21}+\gamma_{31}+\gamma_{32}}}\{B_2^+,B_1^-\}^{\gamma_{21}}\{B_3^+,B_1^-\}^{\gamma_{31}}\{B_3^+,B_2^-\}^{\gamma_{32}}
	\times
	\\&\qquad\times
	[B_1^+,B_2^+,B_3^+]^{\lambda_3}[B_1^+,B_2^+]^{\lambda_2-\lambda_3}(B_1^+)^{\lambda_1-\lambda_2}v_0.
	\end{split}
	\end{equation}
The second expression given in \eqref{sec5_eq_basis-element-identity} gives a basis element corresponding to the tableau $A$ as a polynomial in the parabosonic creation operators $B_j^+$, for $j\in\{1,2,3\}$, acting on the vacuum vector $v_0\in L(p)$. 
To make this explicit we keep the shorthand $\lambda=\lambda_A$ and use \eqref{sec3_eq_def-omega_A}, \eqref{sec3_eq_def-Omega_A} and \eqref{sec3_eq_def-multi-bracket} to write
\begin{equation}
\label{sec5_eq_basis-vector-expression-2}
\begin{split}
\Omega_A&=
	\sum_{\tau\in S_\lambda} \omega_{A^\tau}
	\\&=\sum_{\tau\in S_{\lambda}}[B_{A^\tau(1,1)}^+,\dots,B_{A^\tau(\lambda_1',1)}^+]\cdots [B_{A^\tau(1,\ell(\lambda'))}^+,\dots,B_{A^\tau(\lambda_{\ell(\lambda')}',\ell(\lambda'))}^+]v_0.
	\end{split}
	\end{equation}
Here $A^\tau$ is the row-permuted Young tableau obtained by permuting $A$ with $\tau$ and $A^\tau(k,l)$ is the entry of $A^\tau$ in the $k$'th row and $l$'th column.

To illustrate the formulas \eqref{sec5_eq_basis-vector-expression-1} and \eqref{sec5_eq_basis-vector-expression-2} we consider the following example. If  
\begin{equation}
\ytableausetup{centertableaux,smalltableaux}
A=\ytableaushort{1133,22} ,
\quad\text{ then }\quad 
\gamma_A:=\left(
\begin{matrix}
2&0&0\\
0&2&0\\
2&0&0
\end{matrix}
\right)
\quad \text{ and } \quad 
\lambda_A=(4,2,0).
	\end{equation}
Furthermore,
\begin{equation}
E^{\gamma_A}\Omega_{\lambda_A}
	=E_{31}^2\Omega_{(4,2,0)}
	=12 \{B_3^+,B_1^-\}^2[B_1^+,B_2^+]^2(B_1^+)^2v_0
	\end{equation}
and
\begin{equation}
\begin{split}
\Omega_A
	&=
	\sum_{\tau\in S_{(4,2,0)}} \omega_{A^\tau}
	\\&=
	8\omega_{\ytableaushort{1133,22}}
	+8\omega_{\ytableaushort{1313,22}}
	+8\omega_{\ytableaushort{3113,22}}
	+8\omega_{\ytableaushort{1331,22}}
	+8\omega_{\ytableaushort{3131,22}}
	+8\omega_{\ytableaushort{3311,22}}	
	\\&=
	8[B_1^+,B_2^+]^2(B_3^+)^2v_0
	+8[B_1^+,B_2^+][B_3^+,B_2^+]B_1^+B_3^+v_0
	\\&\qquad\qquad 
	+8[B_3^+,B_2^+][B_1^+,B_2^+]B_1^+B_3^+v_0
	+8[B_1^+,B_2^+][B_3^+,B_2^+]B_3^+B_1^+v_0
	\\&\qquad\qquad 
	+8[B_3^+,B_2^+][B_1^+,B_2^+]B_3^+B_1^+v_0
	+8[B_3^+,B_2^+]^2(B_1^+)^2v_0
	.
	\end{split}
	\end{equation}

In Section \ref{sec4} we studied the Mickelsson-Zhelobenko algebra $Z(\mathfrak{osp}(1|2n),\mathfrak{gl}(n))$, which is $Z(\mathfrak{osp}(1|6),\mathfrak{gl}(3))$ when $n=3$. According to Theorem \ref{sec4_theo_mickelsson-zhelobenko-generator} and \eqref{sec4_eq_mickelsson-zhelobenko-generator-identity-3} this algebra is generated by the raising and lowering operators $z_j^\pm$ and $z_{ij}^\pm$, for $1\leq i\leq j\leq 3$.
The raising and lowering operators $z_1^+$, $z_2^+$, $z_3^+$, $z_1^-$, $z_2^-$ and $z_3^-$ can be written explicitly as follows modulo the ideal $\mathfrak{J}'$.
\begin{equation}
\label{sec5_mickelsson-algebra-generators-explicit-n=3}
\begin{split}
z_1^+&= B_1^+
	\\
z_2^+&= B_2^+(h_1-h_2) -E_{21}B_1^+
	\\
z_3^+&= B_3^+(h_1-h_3)(h_2-h_3) -E_{32}B_2^+(h_1-h_3) -E_{31}B_1^+(h_2-h_3-1) + E_{21}E_{32}B_1^+
	\\
z_1^-&= B_1^-(h_1-h_2)(h_1-h_3) + E_{21}B_2^-(h_1-h_3) +E_{31}B_3^-(h_1-h_2) + E_{21}E_{32}B_3^-
	\\
z_2^-&= B_2^-(h_2-h_3) +E_{32}B_3^-
	\\
z_3^-&= B_3^-
	\end{split}
	\end{equation}

In Section \ref{sec4.3} we used these raising and lowering operators in the proof of Proposition \ref{sec4_prop_h.w.-action}, which gives the actions of the paraboson creation and annihilation operators $B_j^\pm$ on any $\mathfrak{gl}(n)$-highest weight vector $\Omega_\lambda$ as a linear combination of vectors of the form $E^{e_I}\Omega_{\lambda\pm\epsilon_i}$. It was noted that these linear combinations are not in general expansions of $B_j^\pm\Omega_\lambda$ into linear combinations of the vectors from the basis $\{E^{\gamma_A}\Omega_{\lambda_A} :A\in\mathbb{Y}, \ell(\lambda_A)\leq p\}$. The reasons for this were discussed in detail at the end of Section \ref{sec4.3}. 
When $n=3$ we may apply some small modifications to the expansions from Proposition \ref{sec4_prop_h.w.-action} to get proper basis expansions of the vectors $B_j^\pm\Omega_\lambda$. 

\begin{prop}
\label{sec5_prop_action-h.w.-n=3}
Let $n=3$ and $\lambda\in\mathbb{P}$. Then the vectors $B_j^\pm\Omega_\lambda$ have the following expansions into linear combinations of the vectors from the basis $\{E^{\gamma_A}\Omega_{\lambda_A} :A\in\mathbb{Y}, \ell(\lambda_A)\leq p\}$ for $L(p)$.
\begin{equation}
\label{sec5_eq_action-h.w.-n=3-1}
\begin{split}
B_1^+\Omega_\lambda
	&= 
	d_1^+(\lambda)\Omega_{\lambda+\epsilon_1}
	\\
B_2^+\Omega_\lambda
	&= 
	d_2^+(\lambda)\Omega_{\lambda+\epsilon_2} +\frac{d_1^+(\lambda)}{\lambda_1-\lambda_2+1}E_{21}\Omega_{\lambda+\epsilon_1}
	\\
B_3^+\Omega_\lambda
	&=
	d_3^+(\lambda)\Omega_{\lambda+\epsilon_3} + \frac{d_2^+(\lambda)}{\lambda_2-\lambda_3+1}E_{32}\Omega_{\lambda+\epsilon_2} + \frac{d_1^+(\lambda)(\lambda_1-\lambda_2+2)}{(\lambda_1-\lambda_2+1)(\lambda_1-\lambda_3+2)}E_{31}\Omega_{\lambda+\epsilon_1} 
	\\&\quad
	+ \frac{d_1^+(\lambda)}{(\lambda_1-\lambda_2+1)(\lambda_1-\lambda_3+2)}E_{21}E_{32}\Omega_{\lambda+\epsilon_1}
	\\\\
B_1^-\Omega_\lambda	&=
	d_1^-(\lambda)\Omega_{\lambda-\epsilon_1} - \frac{d_2^-(\lambda)}{\lambda_1-\lambda_2+1}E_{21}\Omega_{\lambda-\epsilon_2} -\frac{(1-\delta_{\lambda_1\lambda_2})d_3^-(\lambda)}{\lambda_1-\lambda_3+2}E_{31}\Omega_{\lambda-\epsilon_3}
	\\&\quad
	+ \frac{d_3^-(\lambda)(1+\delta_{\lambda_1\lambda_2}(\lambda_1-\lambda_2+1))}{(\lambda_1-\lambda_2+1)(\lambda_1-\lambda_3+2)}E_{21}E_{32}\Omega_{\lambda-\epsilon_3}
	\\
B_2^-\Omega_\lambda
	&=
	d_2(\lambda)\Omega_{\lambda-\epsilon_2} - \frac{d_3^-(\lambda)}{\lambda_2-\lambda_3+1}E_{32}\Omega_{\lambda-\epsilon_3}
	\\
B_3^-\Omega_\lambda
	&=
	d_3^-(\lambda)\Omega_{\lambda-\epsilon_3}.
	\end{split}
	\end{equation}
Here $\Omega_{\lambda\pm\epsilon_i}:=0$ if $\lambda\pm\epsilon_i\notin\mathbb{P}$.
	\end{prop}
\begin{proof}
The only difference between these expansions and those in Proposition \ref{sec4_prop_h.w.-action} appear in the expansion of $B_1^-\Omega_\lambda$. To obtain the one from the other it is sufficient to note that if $\lambda_1=\lambda_2$, then 
\begin{equation}
\label{sec5_eq_action-h.w.-n=3-proof-1}
E_{31}\Omega_{\lambda-\epsilon_3}
	=E_{32}E_{21}\Omega_{\lambda-\epsilon_3}-E_{21}E_{32}\Omega_{\lambda-\epsilon_3}
	=-E_{21}E_{32}\Omega_{\lambda-\epsilon_3}.
	\end{equation}
Here we got the second identity in \eqref{sec5_eq_action-h.w.-n=3-proof-1} by noting the following identities are satisfied by any $\mathfrak{gl}(n)$-highest weight vector $\Omega_\mu$.
\begin{equation}
E_{i+1,i}^{\mu_i-\mu_{i+1}+1}\Omega_\mu=0, \quad (1\leq i \leq n-1).
	\end{equation}
With this we have proven that the expansions in \eqref{sec5_eq_action-h.w.-n=3-1} are satisfied. It only remains to be proven that the expansions are linear combinations of the vectors from the basis $\{E^{\gamma_A}\Omega_{\lambda_A} :A\in\mathbb{Y}, \ell(\lambda_A)\leq p\}$. This is done using Lemma \ref{sec3_lemm_betweenness-like-condition}. Consider for example the term 
\begin{equation}
\label{sec5_eq_action-h.w.-n=3-2}
-\frac{(1-\delta_{\lambda_1\lambda_2})d_3^-(\lambda)}{\lambda_1-\lambda_3+2}E_{31}\Omega_{\lambda-\epsilon_3}
	\end{equation}
in the expansion of $B_1^-\Omega_\lambda$. The coefficient in this term is zero when $\lambda_1=\lambda_2$, so we only need to prove that $E_{31}\Omega_{\lambda-\epsilon_3}$ is a basis vector when $\lambda_1>\lambda_2$.
By noting that 
\begin{equation}
E_{31}\Omega_{\lambda-\epsilon_3}=E^{(\lambda_1-1)e_{11}+\lambda_2e_{22} +(\lambda_3-1)e_{33}+e_{31}}\Omega_{\lambda-\epsilon_3}
	\end{equation}
we can apply Lemma \ref{sec3_lemm_betweenness-like-condition} to the matrix $\gamma=(\lambda_1-1)e_{11}+\lambda_2e_{22} +(\lambda_3-1)e_{33}+e_{31}$. 
Lemma \ref{sec3_lemm_betweenness-like-condition} tells us that $D(\gamma)$ is a s.s. Young tableau of shape $\lambda-\epsilon_3$ when $\lambda_1>\lambda_2$. Consequently, the term \eqref{sec5_eq_action-h.w.-n=3-2} is zero or a scalar multiple of a basis vector.
A similar line of reasoning can be applied to each term in the expansions \eqref{sec5_eq_action-h.w.-n=3-1}. This proves that the expansions are indeed linear combinations of basis vectors.
	\end{proof}

In Section \ref{sec4.4} we constructed a GZ-basis for the module $L(p)$ and described properties of the transition matrix relating it to the basis $\{E^{\gamma_A}\Omega_{\lambda_A} :A\in\mathbb{Y}, \ell(\lambda_A)\leq p\}$, see Theorem \ref{sec4_theo_GZ-basis} and Theorem \ref{sec4_theo_triangular-transition-matrix}. The elements of the GZ-basis were expressed using the raising operators $z_j^+$ from the Mickelsson-Zhelobenko algebra $Z(\mathfrak{osp}(1|2n),\mathfrak{gl}(n))$ together with the lowering operators $y_{mj}$ from the Mickelsson-Zhelobenko algebras $Z(\mathfrak{gl}(m),\mathfrak{gl}(m-1))$, for $m\in\{2,\dots,n\}$.
The raising and lowering operators generating the algebras $Z(\mathfrak{gl}(2),\mathfrak{gl}(1))$ and $Z(\mathfrak{gl}(3),\mathfrak{gl}(2))$ can be expressed as follows using \eqref{sec4_eq_gl-raising}, \eqref{sec4_eq_gl-lowering} and \eqref{sec4_eq_gl-raising-lowering}.
\begin{equation}
\label{sec5_eq-gl-raising-lowering-identities}
\begin{split}
y_{12}&=E_{12}
	\\
y_{13}&=E_{13}
	\\
y_{23}&=E_{23}(h_2-h_1)+E_{21}E_{13}
	\\
y_{21}&=E_{21}
	\\
y_{31}&=E_{31}(h_1-h_2)+E_{21}E_{32}
	\\
y_{32}&=E_{32}.
	\end{split}
	\end{equation}
It is important to note that these identities hold only modulo an appropriate ideal given in the definition of $Z(\mathfrak{gl}(m),\mathfrak{gl}(m-1))$. 
The appearance of this ideal may be disregarded when we interpret the elements of $Z(\mathfrak{gl}(m),\mathfrak{gl}(m-1))$ as operators acting on the space of $\mathfrak{gl}(m-1)$-highest weight vectors in $L(p)$. This is because the elements of the ideal act as the zero operator.

When $n=3$, the vectors of the GZ-basis for $L(p)$ defined in Theorem \ref{sec4_theo_GZ-basis} are given by
\begin{equation}
v_A:= y_{21}^{\gamma_{21}}y_{31}^{\gamma_{31}}y_{32}^{\gamma_{32}}(z_3^+)^{\lambda_3}(z_2^+)^{\lambda_2}(z_1^+)^{\lambda_1}v_0,
	\end{equation}
for $A\in\mathbb{Y}$ with $\ell(\lambda_A)\leq p$. Here we have used the shorthand notations $\gamma=\gamma_A$ and $\lambda=\lambda_A$.

We can now make two important observations regarding the vectors in the GZ-basis for $L(p)$. These are stated in Proposition~\ref{sec5_prop_transition-matrix-n=3}. Firstly, we can use the identities in \eqref{sec5_eq-gl-raising-lowering-identities} to give explicitly the expansion of the any GZ-basis vector as a linear combination of vectors from the basis $\{E^{\gamma_A}\Omega_{\lambda_A} :A\in\mathbb{Y}, \ell(\lambda_A)\leq p\}$. Secondly, we can use this expansion together with Corollary \ref{sec3_coro_PBW-bases-Omega_A} to explicitly express any GZ-basis vector as a linear combination of the vectors $\Omega_A$. Since each vector $\Omega_A$ is defined in \eqref{sec3_eq_def-Omega_A} as a polynomial of parabosonic creation operators $B_j^+$ acting on the vacuum $v_0$, then the second statement in Proposition~\ref{sec5_prop_transition-matrix-n=3} gives explicit descriptions of the GZ-basis vectors as polynomials of parabosonic creation operators $B_j^+$ acting on the vacuum $v_0$.

\begin{prop}
\label{sec5_prop_transition-matrix-n=3}
Let $n=3$ and $A\in\mathbb{Y}$ with $\ell(\lambda_A)\leq p$. Then 
\begin{align}
v_A
	&= \sum_{\ell=0}^{\gamma_{31}} d(\lambda){\gamma_{31} \choose \ell}(\lambda_1-\lambda_2+2-\gamma_{31}+\gamma_{32}+\ell)_{\gamma_{31}-\ell}E^{\gamma(\ell)}\Omega_\lambda
	\label{sec5_eq_transition-matrix-n=3-1}
	\\
	&= \sum_{\ell=0}^{\gamma_{31}} \frac{d(\lambda)\lambda!}{\diag(\gamma(\ell))!}{\gamma_{31} \choose \ell}(\lambda_1-\lambda_2+2-\gamma_{31}+\gamma_{32}+\ell)_{\gamma_{31}-\ell}\Omega_{D(\gamma(\ell))},
	\label{sec5_eq_transition-matrix-n=3-2}
	\end{align}
where $\gamma(\ell):=\gamma+\ell(e_{21}-e_{31}+e_{32}-e_{22})$ and 
\begin{equation}
d(\lambda):=\prod_{j=1}^3\frac{(-1)^{\frac{\lambda_j(\lambda_j+1)}{2}}}{\lambda_j!j^{\lambda_j}}\prod_{k=0}^{\lambda_j-1}\prod_{\ell=1}^{j-1}(k-\lambda_\ell-j+\ell+1-[k-\lambda_\ell ]_2).
	\end{equation}
Here $D(\gamma(\ell))\in\mathbb{Y}(\lambda)$ when $\ell\leq\lambda_2-\lambda_3-\gamma_{32}$ and $E^{\gamma(\ell)}\Omega_\lambda=0=\Omega_{D(\gamma(\ell))}$ when $\ell>\lambda_2-\lambda_3-\gamma_{32}$.
	\end{prop}
\begin{proof}
To obtain this result we use \eqref{sec4_eq_gl-h.w.-vector-raising-operators} together with \eqref{sec5_eq-gl-raising-lowering-identities} and Corollary \ref{sec3_coro_PBW-bases-Omega_A} to make the following calculation.
\begin{align}
v_A
	&= d(\lambda)y_{21}^{\gamma_{21}}y_{31}^{\gamma_{31}}y_{32}^{\gamma_{32}}\Omega_\lambda
	\\
	&= d(\lambda)E_{21}^{\gamma_{21}}(E_{31}(h_1-h_2)+E_{21}E_{32})^{\gamma_{31}}E_{32}^{\gamma_{32}}\Omega_\lambda
	\\
	&= \sum_{\ell=0}^{\gamma_{31}} d(\lambda){\gamma_{31} \choose \ell}(\lambda_1-\lambda_2+2-\gamma_{31}+\gamma_{32}+\ell)_{\gamma_{31}-\ell}E^{\gamma(\ell)}\Omega_\lambda
	\\
	&= \sum_{\ell=0}^{\gamma_{31}} \frac{d(\lambda)\lambda!}{\diag(\gamma)!}{\gamma_{31} \choose \ell}(\lambda_1-\lambda_2+2-\gamma_{31}+\gamma_{32}+\ell)_{\gamma_{31}-\ell}\Omega_{D(\gamma(\ell))}.
	\end{align}
Using Lemma \ref{sec3_lemm_betweenness-like-condition} we can show that $D(\gamma(\ell))\in\mathbb{Y}(\lambda)$ when $\ell\leq\gamma_{22}-\gamma_{33}=\lambda_2-\lambda_3-\gamma_{32}$. Since $\Omega_\lambda$ is a $\mathfrak{gl}(3)$-highest weight vector, we know that $E_{32}^k\Omega_\lambda=0$ when $\lambda_2-\lambda_3<k$. This implies that $E^{\gamma(\ell)}\Omega_\lambda=0$ when
$\lambda_2-\lambda_3<(\gamma(\ell))_{32}=\gamma_{32}+\ell$, so when $\ell> \lambda_2-\lambda_3-\gamma_{32}$. Since $E^{\gamma(\ell)}\Omega_\lambda$ and $\Omega_{D(\gamma(\ell))}$ are proportional we can conclude that $\Omega_{D(\gamma(\ell))}=0$ when $\ell> \lambda_2-\lambda_3-\gamma_{32}$.
	\end{proof}
This theorem gives explicitly the coefficients $T_{AB}$ appearing in  Theorem \ref{sec4_theo_triangular-transition-matrix}.
To illustrate Proposition~\ref{sec5_prop_transition-matrix-n=3} we apply it to the vectors corresponding to the s.s. Young tableaux of shape $\lambda=(4,2,0)$ and weight $\mu=(2,2,2)$. These tableaux are 
\begin{equation}
\ytableausetup{centertableaux,smalltableaux}
A(1)=\ytableaushort{1133,22},
\quad
A(2)=\ytableaushort{1123,23}
\quad\text{ and }\quad 
A(3)=\ytableaushort{1122,33}.
	\end{equation}
Using \eqref{sec5_eq_transition-matrix-n=3-1} we can express the GZ-basis vectors $v_{A(1)}$, $v_{A(2)}$ and $v_{A(3)}$ as linear combinations of the basis vectors $E^{\gamma_{A(1)}}\Omega_\lambda$, $E^{\gamma_{A(2)}}\Omega_\lambda$ and $E^{\gamma_{A(3)}}\Omega_\lambda$:

\begin{equation}
\begin{split}
v_{A(1)}
	&= -\frac{1}{2}E^{\gamma_{A(1)}}\Omega_\lambda-\frac{1}{2}E^{\gamma_{A(2)}}\Omega_\lambda-\frac{1}{12}E^{\gamma_{A(3)}}\Omega_\lambda
	\\
v_{A(2)}
	&= -\frac{1}{3}E^{\gamma_{A(2)}}\Omega_\lambda-\frac{1}{12}E^{\gamma_{A(3)}}\Omega_\lambda
	\\
v_{A(3)}
	&= -\frac{1}{12}E^{\gamma_{A(3)}}\Omega_\lambda.
	\end{split}
	\end{equation}
Using these identities we obtain the following expansions describing the inverse basis transition:
\begin{equation}
\begin{split}
E^{\gamma_{A(1)}}\Omega_\lambda
	&= -2v_{A(1)}+3v_{A(2)}-v_{A(3)}
	\\
E^{\gamma_{A(2)}}\Omega_\lambda
	&= -3v_{A(2)}+3v_{A(3)}
	\\
E^{\gamma_{A(3)}}\Omega_\lambda
	&= -12v_{A(3)}.
	\end{split}
	\end{equation}
Using \eqref{sec5_eq_transition-matrix-n=3-2} together with \eqref{sec3_eq_def-Omega_A} we can express the GZ-basis vectors $v_{A(1)}$, $v_{A(2)}$ and $v_{A(3)}$ as polynomials of parabosonic creation operators acting on the vacuum $v_0$:

\begin{equation}
\begin{split}
v_{A(1)}
	&= 
	-6\Omega_{A(1)}-\Omega_{A(2)}-2\Omega_{A(3)}
	\\&=
 	-48[B_1^+,B_2^+]^2(B_3^+)^2v_0
 	-16[B_1^+,B_3^+]^2(B_2^+)^2v_0
 	-16[B_2^+,B_3^+]^2(B_1^+)^2v_0
 	 \\&\qquad\qquad 
	-24[B_1^+,B_2^+][B_1^+,B_3^+]B_2^+B_3^+v_0
	-24[B_1^+,B_2^+][B_1^+,B_3^+]B_3^+B_2^+v_0
	\\&\qquad\qquad 
	+24[B_1^+,B_2^+][B_2^+,B_3^+]B_1^+B_3^+v_0
	+24[B_1^+,B_2^+][B_2^+,B_3^+]B_3^+B_1^+v_0
	\\&\qquad\qquad 
	-24[B_1^+,B_3^+][B_1^+,B_2^+]B_2^+B_3^+v_0
	-24[B_1^+,B_3^+][B_1^+,B_2^+]B_3^+B_2^+v_0
	\\&\qquad\qquad 
	+\ 8[B_1^+,B_3^+][B_2^+,B_3^+]B_1^+B_2^+v_0
	+\ 8[B_1^+,B_3^+][B_2^+,B_3^+]B_2^+B_1^+v_0
	\\&\qquad\qquad 
	+24[B_2^+,B_3^+][B_1^+,B_2^+]B_1^+B_3^+v_0
	+24[B_2^+,B_3^+][B_1^+,B_2^+]B_3^+B_1^+v_0
	\\&\qquad\qquad 
	+\ 8[B_2^+,B_3^+][B_1^+,B_3^+]B_1^+B_2^+v_0
	+\ 8[B_2^+,B_3^+][B_1^+,B_3^+]B_2^+B_1^+v_0.
	\end{split}
	\end{equation}
\begin{equation}
\begin{split}
v_{A(2)}
	&= 
	-8\Omega_{A(2)}-2\Omega_{A(3)}
	\\&=
 	-16[B_1^+,B_3^+]^2(B_2^+)^2v_0
 	+16[B_2^+,B_3^+]^2(B_1^+)^2v_0
 	 \\&\qquad\qquad 
	-16[B_1^+,B_2^+][B_1^+,B_3^+]B_2^+B_3^+v_0
	-16[B_1^+,B_2^+][B_1^+,B_3^+]B_3^+B_2^+v_0
	\\&\qquad\qquad 
	-16[B_1^+,B_2^+][B_2^+,B_3^+]B_1^+B_3^+v_0
	-16[B_1^+,B_2^+][B_2^+,B_3^+]B_3^+B_1^+v_0
	\\&\qquad\qquad 
	-16[B_1^+,B_3^+][B_1^+,B_2^+]B_2^+B_3^+v_0
	-16[B_1^+,B_3^+][B_1^+,B_2^+]B_3^+B_2^+v_0
	\\&\qquad\qquad 
	-16[B_2^+,B_3^+][B_1^+,B_2^+]B_1^+B_3^+v_0
	-16[B_2^+,B_3^+][B_1^+,B_2^+]B_3^+B_1^+v_0.
	\end{split}
	\end{equation}
\begin{equation}
\begin{split}
v_{A(3)}
	&= 
	-2\Omega_{A(3)}
	\\&=
 	-16[B_1^+,B_3^+]^2(B_2^+)^2v_0
 	-16[B_2^+,B_3^+]^2(B_1^+)^2v_0
	\\&\qquad\qquad 
	-16[B_1^+,B_3^+][B_2^+,B_3^+]B_1^+B_2^+v_0
	-16[B_1^+,B_3^+][B_2^+,B_3^+]B_2^+B_1^+v_0
	\\&\qquad\qquad 
	-16[B_2^+,B_3^+][B_1^+,B_3^+]B_1^+B_2^+v_0
	-16[B_2^+,B_3^+][B_1^+,B_3^+]B_2^+B_1^+v_0.
	\end{split}
	\end{equation}

\appendix

\section{Appendix}
\label{appA}

In this appendix we present the calculation of the coefficients $d_{j}^\pm(\lambda)$ appearing in Proposition \ref{sec4_prop_h.w.-action}. To do this calculation we shall make use of the following technical lemma which gives certain useful identities in $U(\mathfrak{osp}(1|2n))$.
\begin{lemm}
\label{appA_lemm_technical}
For any $i_1,\dots,i_k\in\{1,\dots,n\}$ write, following~\eqref{sec3_eq_def-multi-bracket},
\begin{equation}
\label{appA_eq_technical-1}
[B_{i_1}^+,\dots,B_{i_k}^+]
	:=
	\sum_{\sigma\in S_k} \sgn(\sigma) B_{i_{\sigma(1)}}^+\dots B_{i_{\sigma(k)}}^+.
	\end{equation}
Then the following identities hold in $U(\mathfrak{osp}(1|2n))$.
\begin{align}
E_{ij}[B_{i_1}^+,\dots,B_{i_k}^+]
	&= 
	[B_{i_1}^+,\dots,B_{i_k}^+]E_{ij}
	+
	\sum_{t=1}^k 
	\delta_{i,i_t}
	[B_{i_1}^+,\dots,B_{i_{t-1}}^+,B_{i}^+,B_{i_{t-1}}^+,\dots,B_{i_k}^+],
\label{appA_eq_technical-4}
	\\
B_i^+[B_{i_1}^+,\dots,B_{i_k}^+]
	&= 
	(-1)^{k+1}
	[B_{i_1}^+,\dots,B_{i_k}^+]B_i^+
	+
	(-1)^k\frac{2}{k+1}
	[B_{i_1}^+,\dots,B_{i_k}^+,B_i^+]
\label{appA_eq_technical-2}
	\\
B_i^-[B_{i_1}^+,\dots,B_{i_k}^+]
	&= (-1)^k[B_{i_1}^+,\dots,B_{i_k}^+]B_i^-
	\nonumber
	\\&\quad+
	\sum_{t=1}^k (-1)^{t-1}2k[B_{i_1}^+,\dots,B_{i_{t-1}}^+,B_{i_{t+1}}^+,\dots,B_{i_k}^+]E_{i_ti}
	\label{appA_eq_technical-3}
	\\&\quad+
	\sum_{t=1}^k \delta_{ii_t}(-1)^t k(k-1)[B_{i_1}^+,\dots,B_{i_{t-1}}^+,B_{i_{t+1}}^+,\dots,B_{i_k}^+].
	\nonumber
	\end{align}

	\end{lemm}

\begin{prop}
Let $\lambda\in\mathbb{P}$ with $\ell(\lambda)\leq p$ and $j\in \{1,\dots,n\}$, then
\begin{equation}
d_{j}^+(\lambda)
	=
	\frac{(-1)^{\sum_{\alpha=j}^n(\alpha+1)(\lambda_\alpha-\lambda_{\alpha+1}+\delta_{\alpha j})}}{(\lambda_j+1)j}
	\bigg(
	\prod_{\ell=1}^{j-1}\frac{\lambda_j-\lambda_\ell-j+\ell+1}{\lambda_j-\lambda_\ell-j+\ell+[\lambda_j-\lambda_\ell]_2}
	\bigg)
	\end{equation}
and
\begin{equation}
d_{j}^-(\lambda)
	=
	\frac{(\lambda_j+1)j(\lambda_j+n-j+[\lambda_j]_2(p-n))}{(-1)^{\sum_{\alpha=j}^n(\alpha+1)(\lambda_\alpha-\lambda_{\alpha+1})}}
	\bigg(
	\prod_{\ell=j+1}^n\frac{\lambda_j-\lambda_\ell-j+\ell-1}{\lambda_j-\lambda_\ell-j+\ell-[\lambda_j-\lambda_\ell]_2}
	\bigg).
	\end{equation}
	\end{prop}
\begin{proof}
The paper \cite{Lievens-Stoilova-VanderJeugt-2008} presents a GZ-basis for the $\mathfrak{osp}(1|2n)$-module $L(p)$.
In this basis the $\mathfrak{gl}(n)$-highest weight vectors are the basis vectors $|m_\lambda)$, for $\lambda\in\mathbb{P}$ with $\ell(\lambda)\leq p$, whose corresponding GZ-pattern has entries 
\begin{equation}
m_{ij}=\lambda_i, \quad (1\leq i \leq j\leq n),
	\end{equation}
in the notation of the paper.	

For any $\lambda\in\mathbb{P}$ with $\ell(\lambda)\leq p$, the vector $|m_\lambda)$ is non-zero and has weight $\lambda+\frac{p}{2}$, so by Proposition \ref{sec3_prop_highest-weight-vector} there must exist a non-zero coefficient $\kappa(\lambda)\in \C$ such that 
\begin{equation}
\Omega_\lambda = \kappa(\lambda)|m_\lambda), \quad  (\lambda\in\mathbb{P}, \ell(\lambda)\leq p).
	\end{equation}	
Making the choice $|m_0)=1=v_0$, it follows that $\kappa(0)=1$. Here $0\in\mathbb{P}$ is the empty partition.

For any $j\in\{1,\dots,n\}$ we denote the coefficients of the vectors $|m_{\lambda\pm\epsilon_j})$ in the GZ-basis expansions of $B_j^\pm|m_\lambda)$ by $c_{j}^\pm(\lambda)$. 
On the other hand, by using the expansions \eqref{sec4_eq_h.w.-action-1} and \eqref{sec4_eq_h.w.-action-2} we can see that the result of applying the projection $L(p)\to L(p)^+$ to $B_j^\pm\Omega_\lambda$ is $d_{j}^\pm(\lambda)\Omega_{\lambda\pm\epsilon_j}$. Here $L(p)^+$ is defined as in Section \ref{sec4.2}.
This implies that $d_{j}^\pm(\lambda)$ is the coefficient of $\Omega_{\lambda\pm\epsilon_j}$ in the expansion of $B_j^\pm\Omega_\lambda$ into a linear combination of elements of the basis $\{E^{\gamma_A}\Omega_{\lambda_A}\}$.
From this it follows that 
\begin{equation}
\label{appA_eq_coefficient-identity}
d_{j}^\pm(\lambda)=\frac{\kappa(\lambda)}{\kappa(\lambda\pm \epsilon_j)}c_{j}^\pm(\lambda),
	\end{equation}
for all $j\in\{1,\dots,n\}$ and $\lambda\in\mathbb{P}$ with $\ell(\lambda)\leq p$.

Using the matrix elements calculated in \cite{Lievens-Stoilova-VanderJeugt-2008}, we get for any $j\in\{1,\dots,n\}$ and $\lambda\in\mathbb{P}$ with $\ell(\lambda)\leq p$, that
\begin{equation}
c_{j}^+(\lambda)
	=
	\bigg(
	\prod_{\ell=1}^{j-1}\frac{\lambda_j-\lambda_\ell-j+\ell+1}{\lambda_j-\lambda_\ell-j+\ell}
	\bigg)^{\frac{1}{2}}
	F_j(\lambda_1,\dots,\lambda_n)
	\end{equation}
and 
\begin{equation}
c_{j}^-(\lambda)
	=
	\bigg(
	\prod_{\ell=1}^{j-1}\frac{\lambda_j-\lambda_\ell-j+\ell}{\lambda_j-\lambda_\ell-j+\ell-1}
	\bigg)^{\frac{1}{2}}
	F_j(\lambda_1,\dots,\lambda_{j-1},\lambda_j-1,\lambda_{j+1}\dots,\lambda_n),
	\end{equation}
where
\begin{equation}
F_j(\lambda_1,\dots,\lambda_n)
	=(-1)^{\lambda_{j+1}+\cdots+\lambda_n}
	(\lambda_j+n+1-j+[\lambda_j+1]_2(p-n))^{\frac{1}{2}}
	\bigg(
	\prod_{\substack{\ell=1,\\ \ell\neq i}}^n\frac{\lambda_j-\lambda_\ell-j+\ell}{\lambda_j-\lambda_\ell-j+\ell+[\lambda_j-\lambda_\ell]_2}
	\bigg)^{\frac{1}{2}}.
	\end{equation}
At this point we note that $c_{j}^-(\lambda)=c_{j}^+(\lambda-\epsilon_j)$. This observation will be valuable later on. For now we begin the initial steps in the calculation of $d_{j}^+(\lambda)$.
To help with these calculation we introduce the partition
\begin{equation}
\mu:= \sum_{\ell=1}^{j-1} (\lambda_\ell-\lambda_j)\epsilon_\ell.
	\end{equation}
	
Recall that $\Omega_\lambda=\lambda!\omega_\lambda=\lambda!\omega_{D(\gamma_\lambda)}$, where $D(\gamma_\lambda)$ is the s.s. Young tableau of shape $\lambda$ with entries $(D(\gamma_\lambda))(k,l)=k$, for $(k,l)\in\lambda$. Together \eqref{sec3_eq_def-omega_A} and \eqref{appA_eq_technical-1} then imply that
\begin{equation}
\label{appA_eq_omega_lambda-omega_mu}
\begin{split}
\Omega_\lambda
	=\lambda!\omega_\lambda
	&=
	\lambda!
	[B_1^+,\dots,B_n^+]^{\lambda_n-\lambda_{n+1}}\cdots [B_1^+,B_2^+]^{\lambda_2-\lambda_3}(B_1^+)^{\lambda_1-\lambda_2}v_0
	\\&=
	\frac{\lambda!}{\mu!}
	[B_1^+,\dots,B_n^+]^{\lambda_n-\lambda_{n+1}}\cdots[B_1^+,\dots,B_j^+]^{\lambda_j-\lambda_{j+1}}\Omega_\mu.
	\end{split}
	\end{equation}

Using the expression for $B_j^+\Omega_\lambda$ given in Proposition \ref{sec4_prop_h.w.-action} together with the identity \eqref{appA_eq_omega_lambda-omega_mu} and the formula \eqref{appA_eq_technical-2} one can obtain the following identity

\begin{equation}
\begin{split}
B_j^+\Omega_\lambda
	=
	\sum_{i=1}^j\sum_{s=1}^{j-i+1}\sum_{I\in\mathcal{I}_{ij}(s)}
	(-1)^{\sum_{\alpha=j}^n(\alpha+1)(\lambda_\alpha-\lambda_{\alpha+1})}
	\frac{\lambda_i-\lambda_j+1}{\lambda_i+1}
	d_I^+(\mu)
	\Omega_{\lambda+\epsilon_i},
	\end{split}
	\end{equation}
where
\begin{equation}
d_I^+(\mu)=d_{i}^+(\mu)\frac{\prod_{\ell\in I^\complement} (\mu_i-\mu_\ell-i+\ell+1)}{\prod_{\ell=i+1}^{j} (\mu_i-\mu_\ell-i+\ell)}.
	\end{equation}
This means in particular that 
\begin{equation}
d_{j}^+(\lambda)=\frac{(-1)^{\sum_{\alpha=j}^n(\alpha+1)(\lambda_\alpha-\lambda_{\alpha+1})}}{\lambda_j+1}d_{j}^+(\mu).
	\end{equation}
To calculate $d_{j}^+(\lambda)$ it is therefore enough to calculate $d_{j}^+(\mu)$. 
Recalling \eqref{appA_eq_coefficient-identity} and using that $c_{j}^-(\mu+\epsilon_j)=c_{j}^+(\mu)$ we get
\begin{equation}
d_{j}^+(\mu)
	= 
	\frac{\kappa(\mu)}{\kappa(\mu+\epsilon_j)} c_{j}^+(\mu)
	=
	\frac{c_{j}^-(\mu+\epsilon_j)}{d_{j}^-(\mu+\epsilon_j)} 
	c_{j}^+(\mu)
	=
	\frac{c_{j}^+(\mu)^2}{d_{j}^-(\mu+\epsilon_j)} 
	\end{equation}

Using \eqref{appA_eq_technical-3} the following calculation gives us $d_{j}^-(\mu+\epsilon_j)$:
\begin{equation}
\begin{split}
B_j^-\Omega_{\mu+\epsilon_j}
	&=
	(\mu+\epsilon_j)!
	B_j^-[B_1^+,\dots,B_j^+][B_1^+,\dots,B_{j-1}^+]^{(\mu+\epsilon_j)_{j-1}-(\mu+\epsilon_j)_j}\cdots (B_1^+)^{(\mu+\epsilon_j)_1-(\mu+\epsilon_j)_2}v_0
	\\&=
	(\mu+\epsilon_j)!
	(-1)^{j+1}j(p-j+1)
	[B_1^+,\dots,B_{j-1}^+]^{\mu_{j-1}-\mu_j}\cdots (B_1^+)^{\mu_1-\mu_2}v_0
	\\&=
	\frac{(\mu+\epsilon_j)!}{\mu!}
	(-1)^{j+1}j(p-j+1)
	\Omega_\mu
	\\&=
	(-1)^{j+1}j(p-j+1)
	\Omega_\mu,
	\end{split}
	\end{equation}
so 
\begin{equation}
d_{j}^-(\mu+\epsilon_j)=(-1)^{j+1}j(p-j+1).
	\end{equation}
We now have everything we need to calculate $d_{j}^+(\lambda)$:
\begin{equation}
\begin{split}
d_{j}^+(\lambda)
	&=
	\frac{(-1)^{\sum_{\alpha=j}^n(\alpha+1)(\lambda_\alpha-\lambda_{\alpha+1})}}{\lambda_j+1}
	d_{j}^+(\mu)
	\\&=
	\frac{(-1)^{\sum_{\alpha=j}^n(\alpha+1)(\lambda_\alpha-\lambda_{\alpha+1})}}{\lambda_j+1}
	\frac{c_{j}^+(\mu)^2}{d_{j}^-(\mu+\epsilon_j)}
	\\&=
	\frac{(-1)^{\sum_{\alpha=j}^n(\alpha+1)(\lambda_\alpha-\lambda_{\alpha+1}+\delta_{\alpha j})}}{(\lambda_j+1)j(p-j+1)}c_{j}^+(\mu)^2
	\\&=
	\frac{(-1)^{\sum_{\alpha=j}^n(\alpha+1)(\lambda_\alpha-\lambda_{\alpha+1}+\delta_{\alpha j})}}{(\lambda_j+1)j(p-j+1)}
	\bigg(
	\prod_{\ell=1}^{j-1}\frac{\mu_j-\mu_\ell-j+\ell+1}{\mu_j-\mu_\ell-j+\ell}
	\bigg)
	\times
	\\&\quad \times
	(\mu_j+n+1-j+[\mu_j+1]_2(p-n))^{\frac{1}{2}}
	\bigg(
	\prod_{\substack{\ell=1,\\ \ell\neq j}}^n\frac{\mu_j-\mu_\ell-j+\ell}{\mu_j-\mu_\ell-j+\ell+[\mu_j-\mu_\ell]_2}
	\bigg)
	\\&=
	\frac{(-1)^{\sum_{\alpha=j}^n(\alpha+1)(\lambda_\alpha-\lambda_{\alpha+1}+\delta_{\alpha j})}}{(\lambda_j+1)j}
	\bigg(
	\prod_{\ell=1}^{j-1}\frac{\lambda_j-\lambda_\ell-j+\ell+1}{\lambda_j-\lambda_\ell-j+\ell+[\lambda_j-\lambda_\ell]_2}
	\bigg).
	\end{split}
	\end{equation}
Similarly we can calculate $d_{j}^-(\lambda)$ as
\begin{equation}
\begin{split}
d_{j}^-(\lambda)
	&=
	\frac{c_{j}^-(\lambda)^2}{d_{j}^+(\lambda-\epsilon_j)}
	\\&=
	\frac{(\lambda_j+1)j(\lambda_j+n-j+[\lambda_j]_2(p-n))}{(-1)^{\sum_{\alpha=j}^n(\alpha+1)(\lambda_\alpha-\lambda_{\alpha+1})}}
	\bigg(
	\prod_{\ell=j+1}^n\frac{\lambda_j-\lambda_\ell-j+\ell-1}{\lambda_j-\lambda_\ell-j+\ell-[\lambda_j-\lambda_\ell]_2}
	\bigg).
	\end{split}
	\end{equation}
	\end{proof}

%

\section*{Acknowledgements}
The authors were supported by the EOS Research Project 30889451. The authors would like to thank Prof. Hendrik De Bie for valuable discussions and Alexis Langlois-Rémillard introducing us to the symbolic manipulation system FORM.

	\end{document}